\documentclass[reqno]{amsart}

\usepackage{amsmath,amssymb}
\usepackage{amsfonts, amscd, epsfig, amsmath, amssymb,enumerate,xfrac}
\usepackage{amstext}
\usepackage{graphicx}
\usepackage[usenames,dvipsnames,svgnames,table]{xcolor}
\newcommand{\A}{\mathcal{A}}
\usepackage[normalem]{ulem}

\usepackage{hyperref}

\usepackage{soul} %%%pacote pra riscar em azul
\setstcolor{blue}

\usepackage[charter]{mathdesign}
\newcommand{\Ind}[1]{\textbf{1}{_{\lbrace{#1}\rbrace}}}    
\usepackage{cancel}
\usepackage{tikz}
\usetikzlibrary{backgrounds}
\numberwithin{equation}{section}
\usetikzlibrary{patterns,fadings,plotmarks}
\usetikzlibrary{arrows,decorations.pathmorphing}
\usetikzlibrary{arrows, shapes, positioning,decorations.markings}
\tikzstyle directed=[postaction={decorate,decoration={markings,
		mark=at position .7 with {\arrow{stealth}}}}]
\usetikzlibrary{calc}
\definecolor{light-gray}{gray}{0.95}
\usepackage{graphicx}
\usepackage{epsfig}
\usetikzlibrary{shapes}
\usepackage[active]{srcltx}%%%%%pesquisa inversa
\newtheorem{theorem}{Theorem}[section]
\newtheorem{lemma}[theorem]{Lemma}
\newtheorem{proposition}[theorem]{Proposition}
\newtheorem{corollary}[theorem]{Corollary}
\newtheorem{remark}[theorem]{Remark}
\newtheorem{definition}{Definition}
\newcommand{\mc}[1]{{\mathcal #1}}

\newcommand{\bb}[1]{{\mathbb #1}}

\newcommand{\<}{\langle}
\renewcommand{\>}{\rangle}
\newcommand{\p}{\partial}
\newcommand{\pfrac}[2]{\genfrac{}{}{}{1}{#1}{#2}}

\newcommand{\at}[2]{\genfrac{}{}{0pt}{}{#1}{#2}}

\DeclareFontFamily{U}{BOONDOX-calo}{\skewchar\font=45 }
\DeclareFontShape{U}{BOONDOX-calo}{m}{n}{
  <-> s*[1.05] BOONDOX-r-calo}{}
\DeclareFontShape{U}{BOONDOX-calo}{b}{n}{
  <-> s*[1.05] BOONDOX-b-calo}{}
\DeclareMathAlphabet{\mathcalboondox}{U}{BOONDOX-calo}{m}{n}
\SetMathAlphabet{\mathcalboondox}{bold}{U}{BOONDOX-calo}{b}{n}
\DeclareMathAlphabet{\mathbcalboondox}{U}{BOONDOX-calo}{b}{n}

\usepackage{graphicx}        % standard LaTeX graphics tool
\usepackage{multicol}        % used for the two-column index
\definecolor{columbiablue}{rgb}{0.61, 0.87, 1.0}

\usepackage{tikz}
\usetikzlibrary{arrows,decorations.pathmorphing,backgrounds,positioning,fit,petri}
\usetikzlibrary{shapes}
\newcommand{\tclock}[5]{
\begin{pgflowlevelscope}{\pgftransformscale{#4}}
\begin{scope}[shift={(#1,#2)}]
\shadedraw [inner color=#3!7!gray, outer color=#3!90!black, 
    line width=0.2pt] (0,0) circle (0.5cm);
\foreach \x in {6,12,...,360} {\draw[line width=0.2pt] (\x:0.40cm) -- (\x:0.45cm);}
\foreach \y in {30,60,...,360} {\draw[line width=0.2pt] (\y:0.35cm) -- (\y:0.45cm);}
{\pgfsetarrowsstart{to}
\draw[line width=0.4pt] (0:0.29cm) -- (0.02,0);
\draw[line width=0.4pt]  (90:0.32cm)--(0,0.02);}
\filldraw[fill=black] (-0.055,0.55) rectangle (0.055,0.6);
\filldraw[fill=black] (-0.015,0.51) rectangle (0.015,0.55);
\draw [line width=0.2pt](0,0.61) circle (0.11cm);
\draw [line width=0.2pt](0,0) circle (0.5cm);
\draw [line width=0.2pt](0,0) circle (0.02cm);
\draw [white,thick,domain=30:45] plot ({#5*0.6*cos(\x)}, {#5*0.6*sin(\x)});
\draw [white,thick,domain=20:55] plot ({#5*0.65*cos(\x)}, {#5*0.65*sin(\x)});
\draw [white,thick,domain=10:65] plot ({#5*0.7*cos(\x)}, {#5*0.7*sin(\x)});
\draw [white,thick,domain=135:150] plot ({#5*0.6*cos(\x)}, {#5*0.6*sin(\x)});
\draw [white,thick,domain=125:160] plot ({#5*0.65*cos(\x)}, {#5*0.65*sin(\x)});
\draw [white,thick,domain=170:115] plot ({#5*0.7*cos(\x)}, {#5*0.7*sin(\x)});
\end{scope}
\end{pgflowlevelscope}
}

\title[Non-equilibrium  and stationary fluctuations for the SSEP with  slow boundary]{Non-equilibrium and stationary  fluctuations\\ for the SSEP with slow boundary}

\author{P. Gon\c{c}alves}
\address{Center for Mathematical Analysis,  Geometry and Dynamical Systems,
Instituto Superior T\'ecnico, Universidade de Lisboa,
Av. Rovisco Pais, 1049-001 Lisboa, Portugal.}
\curraddr{}
\email{patricia.goncalves@math.tecnico.ulisboa.pt}
\thanks{}

\author{M. Jara}
\address{IMPA, Estrada Dona Castorina,  no. 110, Rio de Janeiro, RJ-Brazil }
\curraddr{}
\email{mjara@impa.br}
\thanks{}

\author{O. Menezes}
\address{Center for Mathematical Analysis,  Geometry and Dynamical Systems,
	Instituto Superior T\'ecnico, Universidade de Lisboa,
	Av. Rovisco Pais, 1049-001 Lisboa, Portugal.}
\email{otavio.menezes@tecnico.ulisboa.pt}
\thanks{}

\author{A. Neumann}
\address{UFRGS, Instituto de Matem\'atica e Estat\'istica, Campus do Vale, Av. Bento Gon\c calves, 9500. CEP 91509-900, Porto Alegre, Brasil}
\curraddr{}
\email{aneumann@mat.ufrgs.br}
\thanks{}

\begin{document}

\begin{abstract}
We derive the non-equilibrium   fluctuations of one-dimensional symmetric  simple exclusion processes in contact with slowed stochastic reservoirs which are regulated by a factor $n^{-\theta}$. Depending on the range of $\theta$ we obtain processes with various boundary conditions. Moreover, as a consequence of the previous result we deduce the non-equilibrium stationary fluctuations by using the matrix ansatz method which gives us  information on the stationary measure for the model. The main ingredient to prove these results is the derivation of precise bounds on the two point space-time  correlation function, which are a consequence of precise bounds on the transition probability of some underlying random walks. 
\end{abstract}

\maketitle

\section{Introduction}\label{s1}
The derivation of the non-equilibrium fluctuations around the hydrodynamical profile of general interacting particle systems is a very challenging problem in the field. The main difficulty is the lack of a well developed method which allows one to recover the form of the non-equilibrium space-time correlations of the microscopic model. In many models a uniform bound on the space-time correlations, showing that they vanish as the scaling parameter $n$ goes to infinity, is enough to recover the non-equilibrium fluctuations, but here we analyse a model for which this result is not sufficient and therefore extra work is needed in order to get sharper bounds on the aforementioned correlations.

 In this article we analyse  the symmetric simple exclusion process in contact with stochastic reservoirs and we obtain the non-equilibrium fluctuations when the reservoirs are slow.  The model can be defined as follows. We consider the symmetric simple exclusion process evolving in the discrete set $\Sigma_n=\{1,\cdots, n-1\}$, the bulk, and we superpose this dynamics with a Glauber dynamics at each endpoint of $\Sigma_n$. In the bulk, particles perform continuous time symmetric random walks, under the constraint that two particles cannot occupy the same site at any given time. At the endpoints of the bulk, namely at the sites $1$ and $n-1$,  particles can be created or annihilated at a certain rate, which is slower with respect to the jump rate in the bulk. Note that if we were looking at the symmetric simple exclusion process without the superposition of the Glauber dynamics, then the density of particles $\rho(t,u)$ would be a conserved quantity by the dynamics and it is well known that it evolves according to the heat equation $\partial_t \rho(t,u)=\Delta \rho(t,u)$. Adding the slowed Glauber dynamics at the end points of the bulk allows us to ask about the effects at the level of the partial differential equation and at the level of the fluctuations of the system around the profile $\rho(t,u)$.

To properly define our model, we chose rates of creation given by  $\alpha/n^\theta$ at the site $1$ and $\beta/n^\theta$ at the site $n-1$ and rates of annihilation $(1-\alpha)/n^\theta$ at the site $1$ and $(1-\beta)/n^\theta$ at the site $n-1$. For an illustration of the dynamics see Figure \ref{fig.1}. We observe that the role of the parameters $\alpha, \beta
\in(0,1)$ is to fix the density of the reservoirs, so that when $\beta>\alpha$, the difference of the density in the reservoirs creates a flux in the system. More precisely, if $1\sim\beta>\alpha\sim0$ there is a tendency for particles to get in the bulk from the right reservoir and leave the bulk from the left reservoir. 
%The role of $\theta$ is to slow/fast the reservoir dynamics with respect to the bulk dynamics. 
The parameter $ \theta $ controls the intensity of the interaction between the reservoirs and the bulk.
We also observe that we could take more general rates of annihilation replacing $1-\alpha$ (resp. $1-\beta$) by $\gamma$ (resp. $\delta$), but the results would be exactly the same, only the notation would be heavier and for this reason we stick to this choice of the parameters. 
We note that a simple computation shows that for $\alpha=\beta=\rho$ the Bernoulli product measures $\nu_\rho$ are invariant under the dynamics, which is not the case when  $\alpha\neq \beta$. Nevertheless, in the latter case, by using the matrix ansatz method, the author  in \cite{de_paula}  obtained information on the stationary measure of the system and derived  explicit expressions for the empirical profile and the correlation function, see \eqref{eq:stat_sol} and \eqref{eq:corr_est_stat}.

The hydrodynamic limit for this model was analysed in \cite{bmns}. It is given by the heat equation, but depending on the range of the parameter $\theta$ three different types of boundary conditions appear: when $\theta\in[0,1)$ the density profile $\rho(t,u)$ satisfies Dirichlet boundary conditions, which means that the density profile is fixed as being $\alpha$ (resp. $\beta$) at $0$ (resp. $1$)
$$\rho(t,0)=\alpha \quad \textrm{and}\quad \rho(t,1)=\beta;$$ 
when $\theta=1$ the density profile satisfies a type of linear Robin boundary conditions: $$\partial_u \rho(t,0)=\rho(t,0)-\alpha \quad \textrm{and}\quad \partial_u\rho(t,1)=\beta-\rho(t,1)$$ and when $\theta>1$ the density profile satisfies Neumann boundary conditions $$\partial_u \rho(t,0)=\partial_u\rho(t,1)=0.$$ The hydrodynamic limit, in the case where the reservoirs are fast, was analysed in \cite{BGJO} for a more general exclusion dynamics, which includes the one described above.  There it is shown that, in the case $\theta<0$, the density profile has the same behavior as in the case $\theta\in[0,1)$.

The non-equilibrium fluctuations for this model have been analysed in \cite{FGN_Robin} when $\theta=1$  and in \cite{lmo} when $\theta=0$, and the equilibrium fluctuations have been analysed in \cite{FGNPSPDE} for any  value of $\theta\geq0$. In this paper we close the remaining cases, that is, we obtain the non-equilibrium fluctuations for any value of $\theta\geq 0$ and we only leave open  the case $\theta<0$, the fast case. As a consequence of our result, we also derive the non-equilibrium stationary fluctuations. 

Now we give a word about the proof. This is a natural continuation of the work developed in \cite{FGN_Robin} and for that reason we do not present all the details in the proofs and we refer the interested reader to \cite{FGN_Robin}.
The main difference with respect to \cite{FGN_Robin} is that in the microscopic equations satisfied by the density fluctuation field, there is a boundary term that vanishes identically if one chooses Robin boundary conditions for the test functions. Since in our situation the limiting dynamics has either Neumann or Dirichlet boundary conditions, one can not cancel this term by the choice of the test functions. Therefore, one needs a new argument. The idea is to obtain more refined correlation estimates at the boundary of the system. This turns out to be very demanding, as the proofs of  Proposition \ref{prop:corr_bound} and Lemma \ref{lem:conv_Dyn_mart} show. In particular, one needs to obtain precise estimates on the transition probabilities of some one-dimensional and two-dimensional random walks. These estimates have to be {\em uniform} in the behaviour of the walks at the boundary of the domains and, as a consequence,  they allow to obtain very precise bounds on the space-time correlation function near the boundary.

 Here follows an outline of the paper: in Section \ref{s2} we give the precise definition of the model and state the results. In Section \ref{s3} we prove that the density fluctuation fields converge to solutions of the Ornstein-Uhlenbeck equation \eqref{O.U.} assuming tightness, and in Section \ref{s4} we give the proof of the key result in order to close the equations for the density fluctuation field. Section \ref{sec:tightness} is devoted to the proof of tightness and Section \ref{s6} concerns the proof of the precise estimate on the  correlation functions.

\section{Statement of results}\label{s2}

\subsection{The model}
For $n\geq{1}$, we denote by $\Sigma_n$ the set $\{1,\cdots,n-1\}$. The symmetric simple exclusion process with slow boundary is a Markov process $\{\eta_t:\,t\geq{0}\}$ with configuration space $\Omega_n:=\{0,1\}^{\Sigma_n}$. If $\eta$ denotes a configuration of the state space $\Omega_n$, then $\eta(x)=0$ means that the site $x$ is vacant while $\eta(x)=1$ means that the site $x$ is occupied.  This Markov process can be characterized in terms of its infinitesimal generator $\mc L_{n}$, which we define as follows. Fix the  parameters $\theta\geq 0$ and $\alpha, \beta \in(0,1)$. Given a  function $f:\Omega_n\rightarrow \bb{R}$,
\begin{equation}\label{ln}
\begin{split}
(\mc L_{n}f)(\eta)=
& \Big[\frac{\alpha}{n^\theta}(1-\eta(1))+\frac{(1-\alpha)}{n^\theta}\eta(1)\Big]\Big(f( \eta^1)-f(\eta)\Big)\\
+& \Big[\frac{\beta}{n^\theta}(1-\eta(n-1))+\frac{(1-\beta)}{n^\theta}\eta(n-1)\Big]\Big(f(\eta^{n-1})-f(\eta)\Big)\\
+& \sum_{x=1}^{n-2}\Big(f(\eta^{x,x+1})-f(\eta)\Big) \,,
\end{split}
\end{equation}
where $\eta^{x,x+1}$ is the configuration obtained from $\eta$ by exchanging the occupation variables $\eta(x)$ and $\eta(x+1)$:
\begin{equation*}
(\eta^{x,x+1})(y)=\left\{\begin{array}{cl}
\eta(x+1),& \mbox{if}\,\,\, y=x\,,\\
\eta(x),& \mbox{if} \,\,\,y=x+1\,,\\
\eta(y),& \mbox{otherwise}\,,
\end{array}
\right.
\end{equation*}
 and  for $x\in\{1,n-1\}$, the configuration $\eta^x$, is obtained from $\eta$ by flipping  the occupation  variable $\eta(x)$:
 \begin{equation*}
(\eta^x)(y)=\left\{\begin{array}{cl}
1-\eta(y),& \mbox{if}\,\,\, y=x\,,\\
\eta(y),& \mbox{otherwise.}
\end{array}
\right.
\end{equation*}
The dynamics of this model can be described as follows. In the bulk, particles move according to continuous time random walks, but whenever a particle wants to jump to an occupied site, the jump is suppressed. At the left boundary, particles can be created (resp. removed) at rate $\alpha n^{-\theta}$ (resp. $(1-\alpha) n^{-\theta}$). At the right boundary, particles can be created (resp. removed) at rate $\beta n^{-\theta}$ (resp. $(1-\beta) n^{-\theta}$). We consider the process speeded up in the diffusive time scale $n^2$ so that its generator is given by $n^2\mc L_n$.
Let $\mc D([0,T], \Omega_n)$ be the space of trajectories which are right continuous and with left limits, and taking values in $\Omega_n$. For a measure $\mu_n$ in $\Omega_n$, let $\bb P_{\mu_n}$ be the measure in $ \mc D([0, T], \Omega_n)$
induced by the Markov process with generator $n^2\mc L_n$ and the initial measure $\mu_n$ and
denote by $\bb E_{\mu_n}$ the expectation with respect to $\bb P_{\mu_n}$.
 \begin{figure}
 \begin{center}

 \begin{tikzpicture}[thick, scale=0.85][h!]
 \draw[latex-] (-6.5,0) -- (6.5,0) ;
\draw[-latex] (-6.5,0) -- (6.5,0) ;
\foreach \x in  {-6,-5,-4,-3,-2,-1,0,1,2,3,4,5,6}
\draw[shift={(\x,0)},color=black] (0pt,0pt) -- (0pt,-3pt) node[below] 
{};
%\node[circle,shading=ball,minimum width=0.6cm] (ball) at (0.5,0.3) {};
 \node[ball color=black!30!, shape=circle, minimum size=0.7cm] (A) at (1,0.4) {};

       \node[ball color=pink, shape=circle, minimum size=0.7cm]  at (-6.,1.2) {};
       \node[ball color=pink, shape=circle, minimum size=0.7cm]  at (-6.,0.4) {};
        \node[shape=circle,minimum size=0.7cm] (Q) at (-6.,2) {};
        \node[shape=circle,minimum size=0.7cm] (P) at (-5.,2) {};
        \node[shape=circle,minimum size=0.7cm] (M) at (-6.,1.2) {};
        \node[shape=circle,minimum size=0.7cm] (N) at (-5.,1.2) {};
    \node[ball color=black!30!, shape=circle, minimum size=0.7cm] (C) at (0,0.4) {};
       % \node[fill=black!30!,shape=circle,draw=black,minimum size=0.7cm] (D) at (5.5,0.4) {};

   \node[ball color=black!30!, shape=circle, minimum size=0.7cm] (D) at (-2.,0.4) {};
      \node[ball color=black!30!, shape=circle, minimum size=0.7cm] (L) at (3.,0.4) {};

%        \node[shape=circle,draw=black,minimum size=0.7cm] (L) at (-5.5,0.4) {};
%         \node[fill=pink,shape=circle,draw=black,minimum size=0.7cm] (L) at (-5.5,0.4) {};
%         \node[shape=circle,draw=black,minimum size=0.7cm] (M) at (-5.5,1.2) {};
        \node[shape=circle,minimum size=0.7cm] (Q) at (-6.,2.0) {};
             %\node[shape=circle,minimum size=0.7cm] (M) at (-5.5,1.2) {};
         % \node[shape=circle,minimum size=0.7cm] (N) at (-4.5,1.2) {};
           
%             \node[shape=circle,draw=black,minimum size=0.7cm] (R) at (5.5,0.4) {};
%          \node[shape=circle,minimum size=0.7cm] (S) at (5.5,1.2) {};
           \node[ball color=pink, shape=circle, minimum size=0.7cm] (R) at (6.,0.4) {};
            \node[ball color=pink, shape=circle, minimum size=0.7cm] (S) at (6.,1.2) {};
              \node[ball color=pink, shape=circle, minimum size=0.7cm] (T) at (6.,2.0) {};
               \node[shape=circle,minimum size=0.7cm] (U) at (5,2.0) {};
                \node[shape=circle,minimum size=0.7cm] (V) at (6,3.0) {};
               
                  \node[shape=circle,minimum size=0.7cm] (W) at (5,3.0) {};
               
%           \node[fill=pink,shape=circle,draw=black,minimum size=0.7cm] (T) at (5.5,2.0) {};
              %\node[shape=circle,minimum size=0.7cm] (U) at (4.5,2.0) {};
              
              \path [->] (T) edge[bend right =60] node[above] {$\frac{\beta}{n^\theta}$} (U);
              
                \path [->] (W) edge[bend left=60] node[above] {$\frac{1-\beta}{n^\theta}$} (V);
              
               \path [->] (P) edge[bend right=60] node[above] {${\frac{1-\alpha}{n^\theta}}$} (Q);
               \path [->] (M) edge[bend left=60] node[above] {$\frac{\alpha}{n^\theta}$} (N);
              
    \node[shape=circle,minimum size=0.7cm] (K) at (0,0.4) {};
        \node[shape=circle,minimum size=0.7cm] (G) at (2,0.4) {};
    \node[shape=circle,minimum size=0.7cm] (E) at (-1.,0.5) {};
    %\node[shape=circle,draw=white,minimum size=0.7cm] (F) at (1.5,0.4) {};
 %   \node[shape=circle,draw=white,minimum size=0.7cm] (G) at (6.5,0.4) {};
\path [->] (A) edge[bend right=60,draw=black] node[] {$\times$} (K);
\path [->] (A) edge[bend left=60] node[above] {$1$} (G);
\path [->] (D) edge[bend left=60] node[above] {$1$} (E);
%\path [->] (L) edge[bend right =60] node[above] {$1$} (G);
%       \path [->] (A) edge[bend left=60] node[above] {} (E);
%         \path [->] (A) edge[bend left=60] node[above] {} (F);
%           \path [->] (A) edge[bend left=60] node[above] {$\beta$} (G);

%\tclock{6.9}{-0.8}{pink}{0.8}{0}
%\tclock{-5.6}{-0.8}{pink}{0.8}{0}

\tclock{-6.9}{-0.8}{gray}{0.8}{0}
%\tclock{-8.1}{-0.8}{pink}{0.8}{0}

\tclock{-5.65}{-0.8}{pink}{0.8}{0}
%\tclock{-5.65}{-1.4}{pink}{0.8}{0}
%\tclock{-5.65}{-2.0}{azulzinho}{0.8}{0}

\tclock{-4.4}{-0.8}{pink}{0.8}{0}
%\tclock{-4.4}{-1.4}{pink}{0.8}{0}
%\tclock{-4.4}{-2.0}{azulzinho}{0.8}{0}

\tclock{-3.1}{-0.8}{pink}{0.8}{0}
%\tclock{-3.1}{-1.4}{pink}{0.8}{0}
%\tclock{-3.1}{-2.0}{azulzinho}{0.8}{0}

\tclock{-1.8}{-0.8}{pink}{0.8}{0}
%\tclock{-1.8}{-1.4}{pink}{0.8}{0}
%\tclock{-1.8}{-2.0}{azulzinho}{0.8}{0}

\tclock{-0.6}{-0.8}{pink}{0.8}{0}
%\tclock{-0.6}{-1.4}{pink}{0.8}{0}
%\tclock{-0.6}{-2.0}{azulzinho}{0.8}{0}

\tclock{0.6}{-0.8}{pink}{0.8}{0}
%\tclock{0.6}{-1.4}{pink}{0.8}{0}
%\tclock{0.6}{-2.0}{azulzinho}{0.8}{0}

\tclock{1.85}{-0.8}{pink}{0.8}{0}
%\tclock{1.85}{-1.4}{pink}{0.8}{0}
%\tclock{1.85}{-2.0}{azulzinho}{0.8}{0}

\tclock{3.15}{-0.8}{pink}{0.8}{0}
%\tclock{3.15}{-1.4}{pink}{0.8}{0}
%\tclock{3.15}{-2.0}{azulzinho}{0.8}{0}

\tclock{4.4}{-0.8}{pink}{0.8}{0}
%\tclock{4.4}{-1.4}{pink}{0.8}{0}
%\tclock{4.4}{-2.0}{azulzinho}{0.8}{0}

\tclock{5.65}{-0.8}{pink}{0.8}{0}
%\tclock{5.65}{-1.4}{pink}{0.8}{0}
%\tclock{5.65}{-2.0}{azulzinho}{0.8}{0}

\tclock{6.9}{-0.8}{gray}{0.8}{0}
%\tclock{8.1}{-0.8}{pink}{0.8}{0}
\end{tikzpicture}
\bigskip
\caption{The symmetric simple exclusion with slow boundary.}\label{fig.1}
\end{center}
\end{figure}
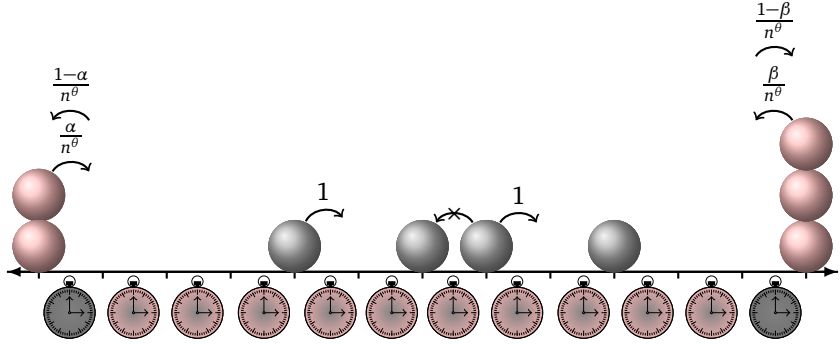

\subsection{Empirical profile and  correlations}

For a measure $\mu_n$ in $\Omega_n$ and for each $x\in\Sigma_n$, we denote by $\rho_t^n(x)$ the empirical profile at the site $x$, given by 
\begin{equation}\label{eq:rho_t}
\rho^n_t(x)\;=\;\mathbb{E}_{\mu_n}[\eta_{tn^2}(x)],
\end{equation}
and at the boundary we set
 $\rho^n_t(0)=\alpha$ and $\rho^n_t(n)=\beta$ for all $t\geq 0$.
A simple computation shows that $\rho_t^n(\cdot)$ is a solution of 
\begin{equation}\label{eq:disc_heat}
\left\{
\begin{array}{ll}
 \partial_t \rho_t^n(x) \;= \; \big(n^2\mathfrak B^\theta_n \rho_t^n\big)(x)\,, \;\; x\in\Sigma_n\,,\;\;t \geq 0\,,\\
 \rho_t^n(0)\;=\alpha\,,  \rho^n_t(n)\;=\;\beta\,,\;\;t \geq 0\,,
\end{array}
\right.
\end{equation}
 where the operator $\mathfrak B_n^\theta$ that acts on functions $f:\overline{\Sigma}_n\to\bb R$ as
\begin{equation}\label{eq:operator_B_n_theta}
(\mathfrak B_n^\theta f)(x)\;=\;\sum_{y=0}^n\xi_{x,y}^{n,\theta}\big(f(y)-f(x)\big)\,, ~~\textrm{ for } x\in \Sigma_n,
\end{equation}
and it is the  infinitesimal generator of the random walk (RW) $\{\mathfrak X^\theta_t,\;t\geq 0\}$ on $\overline{\Sigma}_n$ which is absorbed at the boundary  of $\overline{\Sigma}_n$, that is, at the points $\{0,n\}$. 
Above, $\overline{\Sigma}_n=\Sigma_n\cup \{0,n\}$ and
\begin{equation*}
\xi_{x,y}^{n,\theta}\;=\; \begin{cases}
1\,, & \textrm{ if } \; |y-x|=1 \textrm{ and }x,y\in \Sigma_n\,,\\
n^{-\theta}\,, & \textrm{ if }\; x=1,y=0\textrm{ and } x=n-1\,, y=n\,,\\ 
0\,,& \textrm{ otherwise.}
\end{cases} 
\end{equation*}
The stationary solution of \eqref{eq:disc_heat} is given by
\begin{equation}\label{eq:stat_sol}\rho^n_{ss}(x)=\bb E_{\mu_{ss}}[\eta_{tn^2}(x)]=a_nx+b_n\,,
\end{equation}
where  $$a_n =\tfrac{\beta-\alpha}{2n^\theta+(n-2)} \quad \textrm{and }\quad b_n=a_n({n^\theta}-1)+\alpha.$$
%{\textcolor{red}{
%We assume that
%\begin{equation}\label{eq:ass_ini_prof}
% \max_{ x\in\Sigma_n}\big|\,\rho^n_0(x)-\rho_{ss}^n(x)\big|=O\Big(\tfrac{1}{n\vee n^\theta}\Big).
%\end{equation}
%}}
A simple computation shows that 
 \begin{equation}\label{eq:conve_emp_prof_stat_prof}
\lim_{n\to\infty}\max_{x\in\Sigma_n}\big|\rho^n_{ss}(x)-\bar{\rho}(\tfrac xn)\big|=0,
\end{equation}
where  for $u\in(0,1)$
\begin{equation}\label{eq:hyd_stat_sol}
\bar{\rho}(u)=\left\{
\begin{array}{ll}
(\beta-\alpha)u+\alpha\,;\,\theta<1,\\
\frac{\beta-\alpha}{3}u+\alpha+\frac{\beta-\alpha}{3}\,;\,\theta=1,\\
\frac{\beta+\alpha}{2}\,;\,\theta>1.
\end{array}
\right.
\end{equation}
Now we define the two-point correlation function. 
Let 
\begin{equation}\label{V}
V_n= \{ (x,y)\in \{ 0,\cdots,n\}^{2}: 0<x<y<n \},
\end{equation}
 and its boundary $\partial V_n = \{ (x,y)\in \{ 0,\cdots,n \}^{2}: x=0 \,\, \textrm{or} \,\, y=n \}$.
 
\begin{figure}[!htb]
\centering
\begin{tikzpicture}[thick, scale=0.9]
\draw[->] (-0.5,0)--(5.5,0) node[anchor=north]{$x$};
\draw[->] (0,-0.5)--(0,5.5) node[anchor=east]{$y$};
\begin{scope}[scale=0.75]
\foreach \x in {1,...,4} 
	\foreach \y in {\x,...,4}
		\shade[ball color=black!30!](\x,1+\y) circle (0.3);
		
\foreach \x in {0,...,5} 
	\shade[ball color=pink](\x,6) circle (0.3); 
	 
\foreach \y in {1,...,5} 
	\shade[ball color=pink](0,\y) circle (0.3);  
\end{scope}	
\draw (0,0) node[anchor=north east] {$0$};
\draw (0.75,2pt)--(0.75,-2pt) node[anchor=north] {$1$};
\draw (1.5,2pt)--(1.5,-2pt) node[anchor=north]{$2$};
\draw (2.25,2pt)--(2.25,-2pt);
\draw (3,2pt)--(3,-2pt);
\draw (3.75,2pt)--(3.75,-2pt) node[anchor=north]{\small $n-1$};
%\draw (4.5,-0.05) node[anchor=north]{$n$};
\draw (-0.05,.75) node[anchor=east] {$1$};
\draw (-0.05,1.5) node[anchor=east]{$2$};
\draw (-0.05,3.75) node[anchor= east]{$n-1$};
\draw (-0.05,4.5) node[anchor=east]{$n$};
\end{tikzpicture}
\caption{The set $V_n$ and its boundary $\partial V_n$.}
\end{figure}
For $(x,y)\in V_n$, let $\varphi^n_{t}(x,y)$ denote the two-point correlation function between the occupation sites at $x<y $, which is defined by 
\begin{equation}\label{eq:corr_function}
\varphi_t^n(x,y)=\mathbb{E}_{\mu_{n}}[(\eta_{tn^2}(x)-\rho_{t}^n(x))(\eta_{tn^2}(y)-\rho_{t}^n(y))].
\end{equation}
Doing  simple, but long, computations we see that $\varphi^n_{t}$ is a solution of
\begin{equation}\label{eq:disc_eq_corr}
\begin{cases}
\partial_t \varphi_t^n(x,y)=n^2 \mc A_n^\theta \varphi^n_t(x,y) +g_t^n(x,y), & \textrm{ for } (x,y)\in V_n,\; t>0,\\
\varphi_t^n(x,y)=0, & \textrm{ for } (x,y)\in \partial  V_n, \;t>0,\\
\varphi_0^n(x,y)=\bb E_{\mu_n}[\eta_0(x)\eta_0(y)]-\rho_0^n(x)\rho_0^n(y), & \textrm{ for } (x,y)\in V_n\cup \partial V_n,\\
\end{cases}
\end{equation}
where $\mc A^\theta_n$ is the operator  that acts on functions $f:V_n\cup \partial V_n \to\bb R$ as
 \begin{equation}\label{op_A}
 (\mc A^\theta_nf)(u)=\sum_{v\in V_n} c^\theta_n(u,v)\big[ f(v)-f(u)\big],
 \end{equation}
and it is the infinitesimal generator of the RW $\{\mc X_{tn^2}^\theta; \,t\geq 0\}$ in $V_n\cup \partial V_n$ with jump rates given by $c_n^\theta(u,v)$ and which is absorbed at $\p V_n$. Above, 
 \begin{equation*}
c^\theta_n(u,v)\;=\; \begin{cases}
 1, & \textrm{ if } \; \Vert u-v\Vert =1 \textrm{ and } \;u, v\in  V_n,\\
 n^{-\theta}, & \textrm{ if }\;  \Vert u-v\Vert=1\textrm{ and } u\in V_n,\; v\in \partial V_n,\\ 
 0,& \textrm{ otherwise, }
 \end{cases}
 \end{equation*}
 for $\theta\geq 0$ and   $\Vert \cdot \Vert$ denotes the supremum norm; 
\begin{equation}\label{g}
g^n_{t}(x,y) = -(\nabla_{n}^{+}\rho_{t}^{n}(x))^{2} \delta_{y=x+1}\end{equation} and  $\nabla_{n}^{+}\rho_{t}^{n}(x)= n(\rho_{t}^{n}(x+1)-\rho_{t}^{n}(x)).$
Now we impose some conditions on the initial measures. 
We fix an initial profile $\rho_0:[0,1]\to [0,1]$ which is measurable and of class $C^6$, and we  assume that 
\begin{equation}
\label{assumption 1}
\max_{x\in\Sigma_n}|\rho_0^n(x)-\rho_0(\tfrac xn)|\lesssim \tfrac 1n.
\end{equation}
We observe that the assumption on the regularity of $\rho_0$ is necessary in the proof of Lemma \ref{lem:discrete_gradient}, in order to approximate $\rho_t^n(\cdot)$ by a suitable sequence of functions of class $C^4$.
Above (and in what follows) we write $\psi(x) \lesssim \varphi(x)$ if there exists a constant $C$ independent of $x$ such that $\psi(x) \le C \varphi(x)$ for every $x$.
 Moreover, we also assume that
\begin{equation}\label{eq:ass_ini_corr_1}
\max_{y\in\Sigma_n}|\varphi_0^n(x,y)|\lesssim
\begin{cases}
\tfrac{n^\theta}{n^2}, \; \theta\leq 1,\\
\tfrac {1}{n}, \; \theta\geq 1,
\end{cases}\quad \mbox{for } x=1,n-1,
\end{equation} and that
\begin{equation}\label{eq:ass_ini_corr}
\max_{(x,y)\in V_n}|\varphi_0^n(x,y)|\lesssim\tfrac {1}{n}.
\end{equation} 
\begin{proposition}\label{prop:corr_bound}
Assume \eqref{eq:ass_ini_corr_1} and \eqref{eq:ass_ini_corr}. It holds that
\begin{equation}\label{eq181}
\sup_{t\geq 0}\max_{y\in \Sigma_n}|\varphi_t^n(x,y)|\lesssim
\begin{cases}
\frac{n^\theta}{n^2}, \; \theta\leq1,\\
\frac {1}{n}, \; \theta\geq 1,
\end{cases}\quad \mbox{for } x=1,n-1\,,
\end{equation}
and 
\begin{equation}\label{eq1812}
\sup_{t\geq 0}\max_{(x,y)\in V_n}|\varphi_t^n(x,y)|\lesssim\tfrac{1}{n}.
\end{equation}
\end{proposition}The proof of this proposition is presented in  Section \ref{prova_corr_bound}. 
 In this case, contrarily to the empirical profile, it is quite complicate to obtain an expression for the stationary solution of \eqref{eq:disc_eq_corr}. Nevertheless, we  note that a simple, but long, computation shows that the solution, in the case where the starting measure is the stationary state $\mu_{ss}$, is given by
 \begin{equation}\label{eq:corr_est_stat}
\begin{split}
\varphi_{ss}^n(x,y)=-\frac{(\alpha-\beta)^{2}(x+n^{\theta}-1)(n-y+n^{\theta}-1)}{(2n^{\theta}+n-2)^{2}(2n^{\theta}+n-3)}.
\end{split}
\end{equation}
From the previous identity it follows that
\begin{equation*}\label{eq:bound_stat_corr}
\max_{y\in\Sigma_n}|\varphi_{ss}^n(x,y)| \lesssim
\begin{cases}
\frac{n^\theta}{n^2}, \; \theta\leq 1,\\
\frac {1}{n^\theta}, \; \theta\geq 1,
\end{cases}\quad \mbox{for } x=1,n-1\,,
\end{equation*}
and that
\begin{equation*}\label{eq181_stat}
\max_{(x,y)\in V_n}|\varphi_{ss}^n(x,y)|\lesssim\tfrac{1}{n\vee n^\theta}.
\end{equation*}

\subsection{Stationary measures}

For $\rho\in(0,1)$, let $\nu^n_\rho$ be the Bernoulli product measure in $\Omega_n$ with density $\rho$, that is 
\begin{equation}\label{eq:bernoulli_mea}
\nu^n_\rho\{\eta: \eta(x)=1\}=\rho.
\end{equation}
Under this measure the occupation variables $\{\eta(x)\}_{x\in\Sigma_n}$ are independent and for each $x\in\Sigma_n$ the random variable $\eta(x)$ has Bernoulli distribution of parameter $\rho$.  For  $\alpha=\beta=\rho$  these measures are reversible and, in particular, they are invariant. Nevertheless, when $\alpha\neq \beta$,  since we deal with a finite-state irreducible Markov chain, then there exists a unique stationary measure that we denote by $\mu^n_{ss}.$ A way to get information about this measure is to use the matrix ansatz method introduced in \cite{Derrida1, Derrida2}.  
For that purpose, for a configuration $\eta:=(\eta(1),\cdots,\eta(n-1))$, let  $f_{n-1}(\eta)$ denote the weight of that configuration with respect to the stationary measure $\mu^n_{ss}$ and let us suppose that 
\begin{equation*}\label{eq:matrix_function}
f_{n-1}(\eta) = \textbf{w}^{T} \prod_{x=1}^{n-1}\Big(\eta(x)D+(1-\eta(x))E\Big)\textbf{v},
\end{equation*}
where  $D,E$ are matrices (which, in general, do not commute) and the vectors $\textbf{w}^{T},\textbf{v}$ are present in order to convert the matrix product into a scalar. 
Let $P^n(\eta)$ be the normalized weight of the configuration $\eta$ with respect to the stationary state $\mu^n_{ss}$, which is given by
\begin{equation*}\label{eq:m_a_stat}
P^n(\eta) = \frac{f_{n-1}(\eta)}{Z_{n-1}},
\end{equation*}
where $Z_{n-1}$ is the sum of the weights of the $2^{n-1}$ possible configurations in  $\Omega_n$ which is equal to
$
Z_{n-1} = \textbf{w}^{T}(D+E)^{n-1}\textbf{v}.
$
From the computations of \cite{de_paula}, the matrices $D,E$ and the vectors $\textbf{w}^{T}, \textbf{v}$ satisfy the  following relations:
\begin{equation}\label{eq:algSLOW}
\begin{split}
&DE-ED=D+E:=C,\\
&\textbf{w}^{T}\left[ \tfrac{\alpha}{2n^{\theta}}E-\tfrac{(1-\alpha)}{2n^{\theta}}D \right] = \textbf{w}^{T}, \\ 
&\left[ \tfrac{(1-\beta)}{2n^{\theta}}D - \tfrac{ \beta}{2n^{\theta}}E \right]\textbf{v} = \textbf{v}
\end{split}
\end{equation}
and from this, we can conclude that  \begin{equation*}
Z_{n-1} = \frac{1}{(\alpha-\beta)^{n-1}}\frac{\Gamma(2n^{\theta}+n-1)}{\Gamma(2n^{\theta})},
\end{equation*}
where  $\Gamma(\cdot)$ denotes the Gamma function.
From the previous  information we can get the explicit expressions for the empirical profile, $\rho^n_{ss}(x)$, and the two-point correlation function $\varphi^n_{ss}(x,y)$, these expressions are given in
\eqref{eq:stat_sol}
and in \eqref{eq:corr_est_stat}, respectively.
We refer the interested reader to \cite{de_paula} for more details on how to derive these identities.

\subsection{Hydrodynamic limit}
In \cite{bmns} it was  established the hydrodynamic limit for this model for any  $\theta\geq 0$ and in \cite{BGJO,G_lec_not} it was extended to the case $\theta<0$. For completeness we recall those  results now. 

\begin{definition}\label{def:mea_asso} 

Let $g:[0,1]\to[0,1]$ be a measurable profile.
A sequence $\{\mu_n\}_{n\in \bb N}$ is said to be \textit{associated} to $g(\cdot)$  if,  for any $ \delta >0 $ and any continuous function $ f: [0,1]\to\bb R $ the following limit holds:
\begin{equation*}
\lim_{n\to\infty}
\mu_{n} \Bigg(\, \eta:\, \Big| \frac {1}{n} \sum_{x = 1}^{n\!-\!1} f(\pfrac{x}{n})\, \eta(x)
- \int_0^1 f(u)\,g(u)\, du \Big| > \delta\,\Bigg)=0\,.
\end{equation*}
\end{definition}

\begin{theorem}[Hydrodynamic Limit, \cite{bmns,BGJO,G_lec_not}]\label{theo:HL}
\quad

 Suppose that the 
sequence $\{\mu_n\}_{n\in \bb N}$ is \textit{associated} to a measurable  profile $\rho_0(\cdot)$ in the sense of Definition \ref{def:mea_asso}. 
Then,  for each $ t \in [0,T] $, for any $ \delta >0 $ and any continuous function $ f:[0,1]\to\bb R $, 
\begin{equation*}
\lim_{ n \rightarrow +\infty }
\bb P_{\mu_{n}}\Bigg[\eta_{\cdot} : \Big\vert \frac{1}{n} \sum_{x =1}^{n-1}
f(\pfrac{x}{n})\, \eta_{tn^2}(x) - \int_0^1 f(u)\, \rho(t,u)\, du\, \Big\vert
> \delta \Bigg] \;=\; 0\,,
\end{equation*}
 where  $\rho(t,\cdot)$ is the unique weak solution of the 
 heat equation 
 
 \begin{equation}\label{heat_eq}
\begin{cases}
\p_t \rho(t,u)= \p_u^2 \rho(t,u)\,, & \textrm{for } t>0\,,\, u\in (0,1)\,,\\
\rho(0,u)=\rho_0(u)\,,& u\in [0,1]\,,
\end{cases}
\end{equation}
with the following boundary conditions:
  \begin{enumerate}
  \item For  
$ \theta<1$,  $\rho(t,0) = \alpha$ and  $ \rho(t,1) = \beta$,  for  $t>0$; \label{Dirichlet}
\item For  $\theta=1$, $\p_u \rho(t,0) = \rho(t,0)-\alpha$ and $\p_u \rho(t,1) = \beta-\rho(t,1)$,  for $t>0$;\label{Robin}\item
For $\theta>1$, $\p_u \rho(t,0) = \p_u \rho(t,1) =0$, for $t>0$.\label{Neumann} 
 \end{enumerate}
\end{theorem}

\bigskip

\subsection{Density fluctuations}

In this subsection we state the main results of this article. More precisely,  in Subsection \ref{sec:space_test_functions}  we introduce the space of test functions where the functional associated to the density fluctuations of the system will act, for each regime of $\theta$. Then in Subsection \ref{sec:non_eq_flu} we give the proper notion of the density fluctuation field and in Theorem \ref{non_eq_flu} we state its convergence along subsequences starting from general initial measures, in Theorem \ref{thm27} we state its convergence when assuming that the initial field converges and in Corollary \ref{cor} we state its convergence when the system starts from a local Gibbs state.  
\subsubsection{The space of test functions}
\label{sec:space_test_functions}
The space $ C^\infty([0,1])$ is the space of functions  $f:[0,1]\to \bb R$ such that $f$ is continuous in $[0,1]$ as well as  all its  derivatives. 
\begin{definition}\label{def1} Let $\mathcal S_{\theta}$ denote the set of functions $f\in C^\infty([0,1])$ such that for any $k\in\mathbb{N}\cup \{0\}$ it holds that 

\begin{enumerate}[(1)]
\item for $\theta<1$: 
$ \p_u^{2k} f(0)=\p_u^{2k} f(1)=0$; 
\item for $\theta=1$:
$ \p_u^{2k+1} f(0)= \p_u^{2k}f(0)$ and $\p_u^{2k+1} f(1)=- \p_u^{2k}f(1)
$;
\item for $\theta>1$:
$ \p_u^{2k+1} f(0)=\p_u^{2k+1} f(1)=0.$ 
\end{enumerate}
\end{definition}

\begin{definition}\label{def:laplacian_operator}

For $\theta\geq 0$, let $-\Delta_\theta$ be the positive self-adjoint operator on $L^2{[0,1]}$, defined  on $f\in\mc S_\theta$ by
\begin{equation}\label{laplacian}
\Delta_\theta f(u)\;=\;\left\{\begin{array}{cl}
\partial_u^2 f(u)\,,& \mbox{if}\,\,\,u\in(0,1),\\ 
\partial_u^2 f(0^+)\,,& \mbox{if} \,\,\,u=0,\\
\partial_u^2 f(1^-)\,,& \mbox{if} \,\,\,u=1.\\
\end{array}
\right.
\end{equation}
Above, $\partial_u^2 f(a^\pm)$ denotes the side limits at the point $a$.  Let   $\nabla_\theta: \mc S_\theta\rightarrow C^\infty([0,1])$  be the operator given by 
\begin{equation}\label{nabla}
\nabla_\theta f(u)\;=\;\left\{\begin{array}{cl}
\partial_u f(u)\,,& \mbox{if}\,\,\,u\in(0,1),\\
\partial_u f(0^+)\,,& \mbox{if} \,\,\,u=0,\\
\partial_u f(1^-)\,,& \mbox{if} \,\,\,u=1.\\
\end{array}
\right.
\end{equation}

\end{definition}

\begin{definition}\label{def2}
Let $T_t^\theta:\mc S_\theta\to\mc S_\theta$ be the semigroup associated to \eqref{heat_eq} with the corresponding boundary conditions with $\alpha=\beta=0$. That is, given $f\in \mc S_\theta$, by $T_t^\theta f$ we mean the solution of the homogeneous version of \eqref{heat_eq} with initial condition $f$.
\end{definition}

\begin{definition}
Let $\mc S'_{\theta}$ be the topological dual of $\mc S_{\theta}$ with respect to the topology generated by the seminorms 
 \begin{equation}\label{semi-norm}
\|f\|_{k}=\sup_{u\in[0,1]}|\partial_u^kf(u)|\,,
\end{equation}
where $k\in\mathbb{N}\cup \{0\}$.
That is, $\mc S'_\theta$ consists on  all the linear functionals $f:\mc S_\theta\to \bb R$ which are continuous with respect to all the seminorms $\Vert \cdot \Vert_k$.
\end{definition}

Let  $\mc D([0,T],\mc S'_\theta)$ (resp. $\mc C([0,T], \mc S'_\theta)$) be the space of trajectories which are right continuous and with left limits (resp. continuous) and  taking values in $\mc S'_\theta$.

\subsubsection{Non-equilibrium fluctuations}\label{sec:non_eq_flu}

\label{sec:results_non_eq}

\begin{definition}[Density fluctuation field]
We define the density fluctuation field $\mc Y_\cdot^n$ as the time-trajectory of linear functionals acting on functions $f\in\mc S_\theta$ as
\begin{equation}\label{density field}
\mc Y^n_t(f)\;=\;\frac{1}{\sqrt{n}}\sum_{x=1}^{n-1}f(\tfrac{x}{n})\Big(\eta_{tn^2}(x)-\rho^n_t(x)\Big)\,.
\end{equation}
\end{definition}
Recall that above $\rho_t^n(x)=\bb E_{\mu_n}[\eta_{tn^2}(x)]$.
For each $n\geq 1$, let  $\bb Q_n$ be the probability measure on $\mc D([0,T],\mc S'_\theta)$  induced by the density fluctuation field $\mc Y^n_\cdot$ and the measure $\mu_n$.  
\begin{theorem}[Non-equilibrium fluctuations]\label{non_eq_flu}
\quad

 Suppose that $\rho_0:[0,1]\to[0,1]$ is measurable and of class  $C^6$ and that $\mu_n$ is such that \eqref{eq:ass_ini_corr_1} and  \eqref{eq:ass_ini_corr} hold. Then, the sequence of measures $\{\bb Q_n\}_{ n\in \mathbb{N}}$ is tight on  $\mc D([0,T],\mc S'_\theta)$ and 
all limit points $\bb Q$ are probability measures concentrated on paths $\mathcal{Y}_\cdot$ satisfying 
 \begin{equation}\label{characterization}
\mathcal{Y}_t(f)\;=\;\mathcal{Y}_0(T_t^\theta f)+\mathcal W_t(f)\,,
\end{equation}
for any $f\in\mathcal S_\theta$.
Above $T_t^\theta$ is the semigroup  given  in  Definition \ref{def2} and  $\mathcal W_t(f)$ is a mean-zero Gaussian variable of variance  
\begin{equation}\label{eq212}
 \int_0^t\|\nabla T^\theta_{t-r}  f\|^2_{L^2(\rho_r)}dr\,,
\end{equation}
where for $r>0$ 
\begin{equation}\label{norm}
\begin{split}
\<f, g\>_{L^2(\rho_r)}= \int_0^12\chi(\rho(r,u))\,f(u)g(u)\,du\,, \\
\end{split}
\end{equation}
  $\rho(t,u)$ is the solution of the hydrodynamic equation  \eqref{heat_eq} with the corresponding boundary conditions, and $\chi(u)=u(1-u)$. 
Moreover,  $\mathcal{Y}_0$ and $\mathcal W_t$ are uncorrelated in the sense that $\bb E \Big[\mc Y_0(f) \, \mathcal W_t(g) \Big]\;=\;0$,
for all $f,g\in \mc S_\theta$.
\end{theorem}
\begin{theorem}[Ornstein-Uhlenbeck limit]\label{thm27}
\quad

Assume that the sequence of  initial density fields $\{\mc{Y}_0^n\}_{n\in\bb N}$ converges, as $n\to\infty$, to a mean-zero Gaussian field  $\mc Y_0$ with covariance given on $f,g\in\mc S_\theta$ by
\begin{equation}\label{covar}
\lim_{n\to\infty}\mathbb E_{\mu_n}\Big[\mc Y^n_0 (f)\mc Y^n_0(g)\Big]\;=\;\mathbb E\,\Big[\mc Y_0 (f)\mc Y_0(g)\Big]\;:=\;\sigma(f,g)\,.
\end{equation}
Then, the sequence $\{\bb Q_n\}_{ n\in \mathbb{N}}$ converges, as $n\to\infty$, to a generalized Ornstein-Uhlenbeck (O.U.) process, which is the formal solution of  the equation:
\begin{equation} \label{O.U.}
\partial_t \mathcal{Y}_t\;=\;\Delta_\theta\mathcal{Y}_tdt+\sqrt{2\chi(\rho_t)}\nabla_\theta W_t\,,
\end{equation}
where $ W_t$ is a space-time white-noise of unit variance and $\Delta_\theta$,  $\nabla_\theta$ are given in Definition~\ref{def:laplacian_operator}. As a consequence, the covariance of the limit field $\mathcal{Y}_t$ is given on $f,g\in{\mc S_\theta}$ by
\begin{equation}\label{covariance non eq limit field}
 E\,[\mathcal{Y}_t(f)\mathcal{Y}_s(g)]\;=\;\sigma(T^\theta_tf,T^\theta_sg)+\int_0^s\<\nabla_\theta T^\theta_{t-r} f, \nabla_\theta T^\theta_{s-r}g\>_{L^2(\rho_r)}\,dr\,.
\end{equation}
\end{theorem}

As a consequence of the previous result we obtain the 
non-equilibrium fluctuations starting from a Local Gibbs state.

\begin{corollary}[Local Gibbs state]\label{cor}

\quad

Fix a measurable profile $\rho_0:[0,1]\to[0,1]$  of class  $C^6$ 
and  start the process from a Bernoulli product measure with marginal given by $\nu_{\rho_0(\cdot)}\{\eta:\eta(x)=1\}=\rho_0(\tfrac{x}{n})$.
Then, the previous result is true, but in this case we have, for $f,g\in{\mc S_\theta}$, that
\begin{equation}\label{covariance_local_gibbs}
\sigma(T^\theta_tf,T^\theta_sg)\;=\;\int_0^1 \chi(\rho_0(u))\,T^\theta_tf(u)T^\theta_sg(u)\,du\,,
\end{equation}
where $\rho_0(\cdot)$ is the  initial condition of the hydrodynamic equation \eqref{heat_eq}.
\end{corollary}

\subsection{Non-equilibrium stationary fluctuations}

Now we start the  process from the stationary measure $\mu_{ss}^n$ so that $\rho_{ss}^n(x)=\bb E_{\mu^n_{ss}}[\eta_{tn^2}(x)]$ and the stationary density fluctuation field is acting on functions $f\in \mc S_{\theta}$ as
\begin{equation*}
\mathcal Y^n_t (f) \;=\; \frac 1{\sqrt n} \sum_{x=1}^{n-1} f\big(\tfrac x n\big)
\big( \eta_{tn^2}(x) - \rho^n_{ss}(x)\big)\,. 
\end{equation*}
As above, for each $n\geq 1$, let  $\bb Q^{ss}_n$ be the probability measure on $\mc D([0,T],\mc S'_\theta)$  induced by the density fluctuation field $\mc Y^n_\cdot$ and the measure $\mu^n_{ss}$.   With respect to this starting measure we have that:
\begin{theorem}[Stationary fluctuations: $\theta\neq 1$]\label{theo:flustatheta}

\quad

Suppose to start the process from $\mu_{ss}^n$. Then, $\mathcal{Y}^n$ converges to the centered Gaussian field $\mathcal{Y}$  with covariance given on $f,g\in\mc S_{\theta}$ by:
\begin{equation}\label{stat convariance}
 E[\mathcal{Y}(f)\mathcal{Y}(g)]\;=\;\int_0^1 \chi(\overline{\rho}(u))f(u)g(u)\,du-(\beta-\alpha)^2\int_0^1 [(-\Delta_\theta)^{-1}f(u)]g(u) \,du\,,
\end{equation}
where $\overline{\rho}(\cdot)$ is given in \eqref{eq:hyd_stat_sol}.
\end{theorem}

Note that when $\alpha=\beta=\rho$ the stationary measure is the Bernoulli product measure $\nu_\rho$ and in this case the density fluctuation field is given by
\begin{equation*}
\mathcal Y^n_t (f) = \frac 1{\sqrt n} \sum_{x=1}^{n-1} f\big(\tfrac x n\big)
\big( \eta_{tn^2}(x) - \rho\big)
\end{equation*}
and it converges to a centered Gaussian field $\mathcal{Y}$  with covariance given on $f,g\in\mc S_{\theta}$ by:
\begin{equation}\label{eq:equi_stat convariance}
 E_{\bb Q}[\mathcal{Y}(f)\mathcal{Y}(g)]\;=\;\chi(\rho)\int_0^1 f(u)g(u)\,du\,.
\end{equation}
Last result was obtained in \cite{FGNPSPDE} for all the regimes of $\theta\geq 0$.
We recall that in \cite{FGN_Robin} the stationary fluctuations where derived for the case $\theta=1$ when $\alpha\neq \beta$. The precise statement in that case is  given in the next result.

\begin{theorem}[Stationary fluctuations: $\theta=1$, \cite{FGN_Robin}]\label{theo:flustatheta1}
\quad

Suppose to start the process from $\mu_{ss}^n$ with $\alpha\neq{\beta}$. Then, $\mathcal{Y}^n$ converges to the centered Gaussian field $\mathcal{Y}$  with covariance given on $f,g\in\mc S_{\theta}$ by:
\begin{equation*}%\label{stat convariance_1}
\begin{split}
& E[\mathcal{Y}(f)\mathcal{Y}(g)]\;=\;\int_0^1 \chi(\overline{\rho}(u))f(u)g(u)\,du-\Big(\tfrac{\beta-\alpha}{3}\Big)^2\int_0^1 [(-\Delta_\theta)^{-1}f(u)]g(u) \;du\\
&\quad \quad \quad\quad+\tfrac{2(2\beta+\alpha)(2\beta-1)}{3}\int_0^\infty\hspace{-10pt} T^\theta_t f(1)T^\theta_t g(1)\,dt+\tfrac{2(\beta+2\alpha)(2\alpha-1)}{3}\int_0^\infty\hspace{-10pt} T^\theta_t f(0)T^\theta_t g(0)\,dt,\\
\end{split}
\end{equation*}
where $$\overline{\rho}(u)=\big(\pfrac{\beta-\alpha}{3}\big)\,u +\pfrac{\beta+2\alpha}{3}$$ is the stationary solution of \eqref{heat_eq} with the Robin boundary conditions given  in \eqref{Robin} in the statement of Theorem  \ref{theo:HL}.
\end{theorem}

\section{Proof of Theorem \ref{non_eq_flu}}
\label{s3}

The method of proof of this theorem is classical and it relies on showing tightness of the sequence $\{\bb Q_n\}_{n\in\bb N}$ and the characterization of the limit point. In Section \ref{sec:tightness} we prove tightness and here we characterize the limit points.
For that purpose, fix a test function $f\in \mc S_\theta$. By Dynkin's formula, we have that
\begin{equation}\label{eq:martingales}
\begin{split}
&\mc M^n_t(f)=\mc Y_t^n(f)-\mc Y_0(f)-\int_{0}^t(\partial_s+n^2\mc L_n)\mc Y_s^n(f)\,ds\qquad\mbox{and}\\
&\mc N^n_t(f)=(\mc M_t^n(f))^2-\int_{0}^tn^2\mc L_n \mc Y_s^n(f)^2-2\mc Y_s^n(f)n^2\mc L_n \mc Y_s^n(f)\,ds
\end{split}
\end{equation}
are martingales with respect to the natural filtration $\mc F_t:=\sigma(\eta_s:\, s\leq t)$.
A long, but elementary, computation shows that
\begin{equation}\label{int part of mart}
\begin{split}
(\partial_s+n^2\mc L_n)\mc Y_s^n(f)=&\frac{1}{\sqrt{n}}\sum_{x=1}^{n-1}\Delta_n f\Big(\tfrac{x}{n}\Big) \Big(\eta_{sn^2}(x)-\rho^n_{s}(x)\Big)\\
+& \sqrt n \nabla_n^+f(0)\bar\eta_{sn^2}(1)-\sqrt n \nabla_n^-f(1)\bar\eta_{sn^2}(n-1)\\
-&\frac{n^{3/2}}{n^\theta}f\Big(\tfrac{1}{n}\Big)\bar \eta_{sn^2}(1)-\frac{n^{3/2}}{n^\theta}f\Big(\tfrac{n-1}{n}\Big)\bar\eta_{sn^2}(n-1).
\end{split}
\end{equation}\\

Above $\bar\eta_{sn^2}(x)=\eta_{sn^2}(x)-\rho^n_s(x)$.
On the other hand, doing a simple computation we get that
\begin{equation}\label{quad var}
\begin{split}
\mc L_n \mc Y_s^n(f)^2-2\mc Y_s^n(f)n^2\mc L_n \mc Y_s^n(f)=&\frac{1}{n}\sum_{x=1}^{n-2}\Big(\nabla^+_n f\Big(\frac{x}{n}\Big)\Big)^2 \Big(\eta_s(x)-\eta_s(x+1)\Big)^2\\
+& \frac {n} {n^\theta} \, \Big(f\big(\pfrac{1}{n}\big)\Big)^2\Big(\alpha-2\alpha \eta_s(1)+\eta_s(1)\Big)\\
+&\frac {n}{ n^\theta}\, \Big(f\big(\pfrac{n-1}{n}\big)\Big)^2\Big(\beta-2\beta \eta_s(n-1)+\eta_s(n-1)\Big).
\end{split}
\end{equation}

Above
\begin{equation}\label{discrete_laplacian}
\Delta_nf(x)=n^2\Big[f\Big(\tfrac{x+1}{n}\Big)+f\Big(\tfrac{x-1}{n}\Big)-2f\Big(\tfrac{x-1}{n}\Big)\Big],
\end{equation}
\begin{equation*}
\nabla_n^+f(x)=n\Big[f\Big(\tfrac{x+1}{n}\Big)-f\Big(\tfrac{x}{n}\Big)\Big]
\end{equation*}
and
\begin{equation*}
\nabla_n^-f(x)=n\Big[f\Big(\tfrac{x}{n}\Big)-f\Big(\tfrac{x-1}{n}\Big)\Big].
\end{equation*}
\\
Now we fix $t\in [0, T]$  and we consider the  process restricted  to the time interval $[0, t]$. Taking the time-dependent test function $f(s,u)=T^\theta_{t-s}f(u)$, we  can rewrite 
\eqref{int part of mart} as 
\begin{equation}\label{eq:dynkin_new}
\begin{split}
(\partial_s+n^2\mc L_n)\mc Y_s^n(T^\theta_{t-s}f)\;=\;&\mc Y_s^n(\Delta_nT^\theta_{t-s}f-\Delta T^\theta_{t-s}f)+\mc Y_s^n(\Delta T^\theta_{t-s}f-\partial_sT^\theta_{t-s}f)\\
&+ \sqrt n \nabla_n^+T^\theta_{t-s}f(0)\bar\eta_{sn^2}(1)-\sqrt n \nabla_n^-T^\theta_{t-s}f(1)\bar\eta_{sn^2}(n-1)\\
&-\frac{n^{3/2}}{n^\theta}T^\theta_{t-s}f(\tfrac 1n)\bar \eta_{sn^2}(1)-\frac{n^{3/2}}{n^\theta}T^\theta_{t-s}f(\tfrac{n }{n-1})\bar\eta_{sn^2}(n-1).
\end{split}
\end{equation}
Since $T^\theta_{t-s}f$ is smooth  and solves \eqref{heat_eq}, it is easy to show that the first and second terms at the right hand side of last identity vanish, as $n\to+\infty$. 
Now we analyse the remaining terms on the right hand-side of last identity for each regime of $\theta\neq 1$. 
We start with the case $\theta<1$. In this regime the space of test functions is such that the test functions vanish at the boundary of $[0,1]$,  so that the  terms on the second and third lines at the right hand-side of \eqref{eq:dynkin_new} can be rewritten as 
\begin{equation*}
\begin{split}
& \sqrt n \nabla_n^+T^\theta_{t-s}f(0)\bar\eta_{sn^2}(1)-\sqrt n \nabla_n^-T^\theta_{t-s}f(1)\bar\eta_{sn^2}(n-1)\\
-&\frac{\sqrt n}{n^\theta}\nabla_n^+T^\theta_{t-s}f(0)\bar \eta_{sn^2}(1)-\frac{\sqrt n}{n^\theta}\nabla_n^-T^\theta_{t-s}f(1)\bar\eta_{sn^2}(n-1).
\end{split}
\end{equation*}
In  Lemma \ref{lem:conv_Dyn_mart} we prove that the time integral of the previous terms  vanish in $\bb L^2(\bb P_{\mu_n})$, as $n\to+\infty$.
In the case $\theta>1$, the space of test functions is composed of  functions that have first spatial derivative equal to zero at the boundary of $[0,1]$. Therefore, the terms on the second and third lines on the right hand-side of \eqref{eq:dynkin_new}  are equal to
\begin{equation*}
\begin{split}
-\frac{n^{3/2}}{n^\theta}T^\theta_{t-s}f(\tfrac 1n)\bar \eta_{sn^2}(1)-\frac{n^{3/2}}{n^\theta}T^\theta_{t-s}f(\tfrac{n-1}{n})\bar\eta_{sn^2}(n-1)
\end{split}
\end{equation*}
plus a term of order $O(n^{-1/2})$. In  Lemma \ref{lem:conv_Dyn_mart} we prove that the time integral of  last terms also vanishes in $\bb L^2(\bb P_{\mu_n})$, as $n\to+\infty$.

Finally, from the next lemma, it follows  that the sequence of martingales also converges.

\begin{lemma}\label{lem:conv_mart}
For $\phi\in\mc S_\theta$, the sequence of martingales $\{\mc M^n_t(\phi);t\in [0,T]\}_{n\in \bb N}$ converges in the topology of $\mc D([0,T], \bb R)$, as $n\to\infty$, towards a mean-zero Gaussian process $\mathcal W_t(\phi)$ with  quadratic variation given by 
\begin{equation}\label{36}
\begin{split}
  \int_0^t \int_0^12\chi(\rho(r,u)) \big(\nabla_\theta \phi(u)\big)^2du\, dr,
\end{split}
\end{equation}
where $\rho(t,u)$ is the solution of \eqref{heat_eq} with the corresponding boundary conditions. 
\end{lemma}
We do not present the proof of this lemma here, since it is exactly the same as the proof of Lemma 4.1 in \cite{FGN_Robin}.

\section{Probability estimates}
\label{s4}

In this section we prove the following result which is the  key point in order to close the integral part of the martingale in \eqref{int part of mart}.

\begin{lemma}\label{lem:conv_Dyn_mart}
For   $x\in\{1,n-1\}$ and $t\in[0,T]$ it holds that
$$\bb E_{\mu_n} \Big[\Big(\int_0^t C_n^\theta(\eta_{sn^2}(x)-\rho_{s}^n(x))\, ds\Big)^2\Big]\lesssim (C_n^\theta)^2 \tfrac {n^\theta}{ n^2},$$
and as a consequence 
$$\lim_{n\to+\infty}\bb E_{\mu_n} \Big[\Big(\int_0^t C_n^\theta(\eta_{sn^2}(x)-\rho_{s}^n(x))\, ds\Big)^2\Big]=0,$$
for $C_n^\theta=\sqrt n \textbf{1}_{\{\theta<1\}}+n^{3/2-\theta}\textbf{1}_{\{\theta>1\}}.$
\end{lemma}

\begin{proof}
By developing the square in the expectation, we have 
$$\bb E_{\mu_n} \Big[\Big(\int_0^t C^\theta_n(\eta_{sn^2}(x)-\rho_{s}^n(x))\, ds\Big)^2\Big]=2(C_n^\theta)^2\int_0^t\int_0^r\phi_{s,r}^n(x,x)\,\, ds\, dr,$$
where  for $x,y\in\Sigma_n$
\begin{equation}\label{eq:space_time_corr_function}
\phi_{s,r}^n(x,y)=\mathbb{E}_{\mu_n}[(\eta_{sn^2}(x)-\rho_{s}^n(x))(\eta_{rn^2}(y)-\rho_{r}^n(y))].
\end{equation}
Fix a time $s\in[0,T]$ and $x\in\Sigma_n$, and let $$\phi(r,y):=\phi_{s,r}^n(x,y)$$ for $ r\geq s$ and $ y\in \Sigma_n.$
A simple computation shows that $\phi(r,y)$ is a solution of 
\begin{equation*}%\label{eq:space_time_corr_eq}
\left\{
\begin{array}{ll}
 \partial_t \phi(t,y) \;= \;n^2\mathfrak B^\theta_n \phi(t,y) \,, \;\; y\in\Sigma_n\,,\;\;t>s\\
\phi(s,y)\;=\varphi^n_{s}(x,y)\,, \;\; x\neq y,\\
\phi(s,y)\;=\chi(\rho^n_{s}(x))\,, \;\; x=y,
\end{array}
\right.
\end{equation*}
where $\varphi^n_{s}(\cdot, \cdot)$ was defined in \eqref{eq:corr_function}, $\rho^n_{s}(\cdot)$ was defined in \eqref{eq:rho_t}, $\chi(u)=u(1-u)$ and the operator $ \mathfrak  B^\theta_n $ was defined in \eqref{eq:operator_B_n_theta}.
Moreover, the solution of last equation can be expressed in terms of the fundamental solution of the next equation. 
Fix $x\in\Sigma_n$ and let $P^{n,\theta}_{t}(x,y)$ be the solution of
\begin{equation}\label{eq:space_time_corr}
\left\{
\begin{array}{ll}
 \partial_t P_t^{n,\theta}(x,y) \;= \;n^2\mathfrak  B^\theta_nP_t^{n,\theta}(x,y)\,, \;\; y\in\Sigma_n\,,\;\;t >  0\,,\\
P_0^{n,\theta}(x,y)\;=\delta_0(x-y),\;\;y\in\Sigma_n,
\end{array}
\right.
\end{equation}
where $\delta_0(x)=1$ if $x=0$, otherwise it is equal to $0$.
Then, for any $r\geq s$, we have
$$\phi_{s,r}^n(x,y)\,=\,\sum_{z\neq x}P_{r-s}^{n,\theta}(y,z)\varphi_{s}^n(x,z)+P_{r-s}^{n,\theta}(y,x)\chi(\rho^n_{s}(x))\,.$$

Let us now look at the case $x=1$, but we note that the case $x=n-1$ is similar.  From the computations above, we need to evaluate
\begin{equation*}
\begin{split}
&2(C_n^\theta)^2\int_0^t\int_0^r\phi_{s,r}^n(1,1)\, ds\, dr\\
=&2(C_n^\theta)^2\int_0^t\int_0^r\Big\{\sum_{z\neq 1}P_{r-s}^{n,\theta}(1,z)\varphi_{s}^n(1,z)+P_{r-s}^{n,\theta}(1,1)\chi(\rho^n_{s}(1))\Big\}\, ds\, dr\,.
\end{split}
\end{equation*}
Since $\sum_{z\neq 1}P_{r-s}^{n,\theta}(1,z)$ and $\chi(\rho^n_{s}(1))$  are both bounded by one, uniformly on time and on $n$,  and since from Proposition  \ref{prop:corr_bound} we have that
\begin{equation}\label{bound_1}
\sup_{z\neq 1}|\varphi_{s}^n(1,z)|\lesssim \begin{cases}
\frac{n^\theta}{n^2}, \; \theta<1,\\
\frac{1}{n}, \; \theta>1,
\end{cases}
\end{equation}
the proof ends as long as we show that:
\begin{equation}\label{bound_2}
\int_0^t\int_0^rP_{r-s}^{n,\theta}(1,1)\, dr\, ds\lesssim t \frac{n^\theta}{n^2}.
\end{equation}
The previous bound is obtained combining   Lemma \ref{lem:bound_p_theta} 
and Lemma \ref{int_estimate_lemma},  which are  proved  in the next two subsections.
\end{proof}

\subsection{The one-dimensional coupling}\label{1d_coupling_section}

In this subsection we want to compare the fundamental solution of \eqref{eq:space_time_corr} with the fundamental solution of the same equation for $\theta=0$. For that purpose, recall that
 $\{\mathfrak X^\theta_t,\;t\geq 0\}$ is the RW on $\bar{\Sigma}_n$ with  infinitesimal generator $\mathfrak  B_n^\theta $, defined in \eqref{eq:operator_B_n_theta}, which is absorbed at the boundary $\{0,n\}$. 
 
 Let $P_t^{n,\theta}(y,z)$ be the transition probability for this RW,  that is,   $$P_t^{n,\theta}(y,z)=\bb P_y[\mathfrak X_{tn^2}^\theta=z]=\bb P[\mathfrak X_{tn^2}^\theta=z\mid 	\mathfrak X^\theta_0=y].$$
The goal of this subsection is to prove the following lemma.
\begin{lemma}\label{lem:bound_p_theta}
Let  $P_t^{n,\theta}(y,z)$ be defined as above. Then
\begin{equation*}%\label{bound_coupling}
	P_t^{n,\theta}(y,z)\;\leq\;\,n^\theta \,\Big\{P_t^{n,0}(1,z)+P_t^{n,0}(n-1,z)\Big\}+\Big\{	P_t^{n,0}(y,z)-\big(  	P_t^{n,0}(1,z)+ 	P_t^{n,0}(n-1,z)\big)\Big\}\,,
\end{equation*}
for all  $\theta\geq 0$ and $t\geq 0$. 
In particular,
\begin{equation}\label{bound_coupling}
	P_t^{n,\theta}(x,x)\;\leq\;\,n^\theta \,\Big\{P_t^{n,0}(1,x)+P_t^{n,0}(n-1,x)\Big\}\,,\quad\mbox{for} \quad x=1,n-1\,.
\end{equation}
\end{lemma}
\begin{proof}
This result is proved by means of a  coupling argument similar to the one presented in Section 3 of \cite{bmns}. More precisely,  we construct another RW $\{Z_t^\theta=(Y_t^\theta;N_t^\theta),\,t\geq 0\}$ taking values in ${\Sigma}_n\times \bb N$ such that $Y_t^\theta$, its projection in ${\Sigma}_n$, has the same law of the process $\mathfrak X^\theta_t$.
	The walker  now is the process $Z_t^\theta$ and  it walks in different levels of  ${\Sigma}_n$, that is, when it walks in the  level $k$ it is walking in  ${\Sigma}_n\times \{k\}$. 
	In order to clarify the construction of $Z_t^\theta$, we start by saying that it is a coupling of a random quantity of copies of $\mathfrak  X_t^0$, where $\mathfrak  X_t^{0}$ is the RW $\mathfrak X_t^{\theta}$ with $\theta=0$. This is done in a such way that
	at each level
	the law of the walker is the same law of $\mathfrak X_t^0$.
	Then, since the random variable $N_t^\theta$ is telling us  in which level the walker is walking, we have that the law of 
$Y_t^\theta\,\textbf{1}_{N_t^\theta=k}$ is equal to the law of $\mathfrak X_t^0\,\textbf{1}_{N_t^\theta=k}$, for each $k$ fixed.
The dependence on $\theta$ comes from the random number of copies of $X_t^0$.
	
	The walker starts from $(y; k=1)$ following  a realization of $\{\mathfrak X^0_t,\,t\geq 0\}$   on ${\Sigma}_n$ starting at the site $y$. The walks $ \mathfrak X^0_t $ and $ Z^\theta_t $ coincide up to the first jump attempt from 1 to 0 or from $ n-1 $ to $ n $.
	Let us explain this difference: when the walker $Z_t^\theta$ is on level $1$ and at site $1$, and the RW $\mathfrak X_t^0$  jumps to $0$, the walker $Z_t^\theta$ flips an independent coin with probability $n^{-\theta}$ of getting a head and does the following:
If it comes up a head, the walker $Z_t^\theta$ jumps to $0$ (together with  $\mathfrak X_t^0$) and it is absorbed.
 If it comes up a tail, since the walker $Z_t^\theta$ is at the point $(x=1;k=1)$, it jumps to  $(x=1;k=2)$ and re-starts following an independent copy of $\{\mathfrak X^0_t,\,t\geq 0\}$  on ${\Sigma}_n$ starting from $x=1$ on the level $k=2$.
A similar situation occurs when  a copy of $\mathfrak X^0_t$
	jumps from $n-1$ to $n$, for example on the level $k$. In this case, the walker $Z_t^\theta$ flips another independent coin with probability $n^{-\theta}$ of getting a head and it does the following:
If it comes up  a head, $Z_t^\theta$ jumps to $n$  and  it is absorbed.
If it comes up  a tail, since the walker $Z_t^\theta$ is at $(y=n-1;k)$, it  jumps to  $(y=n-1;k+1)$ and moves as another independent copy of $\{\mathfrak X^0_t,\,t\geq 0\}$  on ${\Sigma}_n$ starting from $y=n-1$ on the level $k+1$.
While the RW $Z_t^\theta$ is not absorbed, every time 
	a copy  of $X_t^0$ jumps to $0$ or to $n$, $Z_t^\theta$ flips another independent coin and repeats the procedure described above. To summarize, when the walker $Z_t^\theta$ tries to jump to $0$ or  to $n$, either it is absorbed or it moves to the next level.

	There is another important point to highlight for the RW $\{Z_t^\theta,\,t\geq 0\}$: if  the process $Z_t^\theta $ is at level $i$ it means that  it 
	flipped $i-1$ independent coins and got $i-1$ tails. In other words, consider 
	$\{\zeta_{j},\,j \geq 1\}$
	a sequence of independent and identically distributed Bernoulli($n^{-\theta}$) random variables and $\zeta=\inf\{j : \zeta_{j}=1\}$. Note that $\zeta$ is a Geometric$(n^{-\theta})$ random variable. 
	Thus,
	\begin{equation*}\begin{split}
	\bb P_y[\mathfrak X_t^\theta=z]\;
=\;&\sum_{k=1}^\infty\sum_{i=1}^k\bar{\bb P}_{(y;1)}[Z_t^\theta=(z;i), \, \zeta= k]
	=\;\sum_{k=1}^\infty\sum_{i=1}^k\bar{\bb P}_{(y;1)}[ Y^\theta_t=z,  \,N^\theta_t=i,\, \zeta= k]\\
	\leq\;&\sum_{k=1}^\infty P[ \zeta= k]\sum_{i=1}^k  \bb P_{y_i}[\mathfrak X_t^0=z]\,. \\
	\end{split}
	\end{equation*}
Above $\bar{\bb P}_{(y;1)}$ is the probability induced by the RW $Z_t^\theta$ starting from $y$ at level $1$ and the random variable $ \zeta$, which has marginal distribution $$P[ \zeta= k]=(1-n^{-\theta})^{k-1}\,n^{-\theta}.$$
The points $y_i$ are saying  where the RW $Z_t^\theta$ starts at the level $i$, then $y_1=y,y_2\in\{1,n-1\},\dots,y_k\in\{1,n-1\}$. Thus, for all $t\geq 0$, 
\begin{equation*}\begin{split}
\bb P_y[\mathfrak X_t^\theta=z]\;&\leq \;
\sum_{k=1}^\infty  P[ \zeta= k]\,\Big\{\bb P_{y}[\mathfrak X_t^0=z]\,+\,(k-1)\,\big(  \bb P_{1}[\mathfrak X_t^0=z]+ \bb P_{n-1}[\mathfrak X_t^0=z]\big)\Big\}\\
%&=\Big\{  \bb P_{1}[X_t^0=z]+ \bb P_{n-1}[X_t^0=z]\Big\}\sum_{k=1}^\infty  P[ \zeta= k]\,k\\
%&+\Big\{\bb P_{y}[X_t^0=z]-\big(  \bb P_{1}[X_t^0=z]+ \bb P_{n-1}[X_t^0=z]\big)\Big\}\sum_{k=1}^\infty  P[ \zeta= k]
%\\
&= \,\Big\{  \bb P_{1}[\mathfrak X_t^0=z]+ \bb P_{n-1}[\mathfrak X_t^0=z]\Big\}\,n^{\theta}\\&+\Big\{\bb P_{y}[\mathfrak X_t^0=z]-\big(  \bb P_{1}[\mathfrak X_t^0=z]+ \bb P_{n-1}[\mathfrak X_t^0=z]\big)\Big\}\,. \\
%	=\;& p_t^{n,0}(1,1) \sum_{k=1}^\infty\Big(1-\frac{1}{n^\theta}\Big)^{k-1}\\
%	=\;& p_t^{n,0}(1,1)\; n^\theta.\\
\end{split}
\end{equation*}

	\end{proof}
	
	\subsection{Estimate for the integral of  the solution of \eqref{eq:space_time_corr} with $\theta=0$}
		Note that \eqref{eq:space_time_corr} with $\theta=0$ can be rewritten as 
	\begin{equation}\label{eq:space_time_corr_theta=0}
	\left\{
	\begin{array}{ll}
	\partial_t P_t^{n,0}(x,y) \;= \;\Delta_n P_t^{n,0}(x,y)\,, \;\; y\in\Sigma_n\,,\;\;t >  0\,,\\
	P_0^{n,0}(x,y)\;=\delta_0(x-y),\;\;y\in\Sigma_n\,,
	\end{array}
	\right.
	\end{equation}
	because, in this case, $n^2\mathfrak  B^0_n$ is equal to the discrete one-dimensional Laplacian, $\Delta_n$.
	The goal of this subsection is  to prove the next result.
	\begin{lemma}\label{int_estimate_lemma}
	For all $n\geq 1$, $t\geq 0$ and for $x=1,n-1$,   we have
	  \begin{equation}\label{int_estimate}
\int_0^t\int_0^rP_{r-s}^{n,0}(x,1)\,ds\,dr\lesssim \frac{ t}{n^2}.
	\end{equation}
	
	\end{lemma}
	\begin{proof}

	 Here we  consider the domain of the infinitesimal generator $\mathfrak  B_n^0 $ which was defined in  \eqref{eq:operator_B_n_theta} with  $\theta=0$, as the set $$\mc D(\mathfrak  B_n^0 )=\{f:\overline{\Sigma}_n\to\bb R\,;\; f(0)=0\;\mbox{ and }\;f(n)=0\}.$$
	For $\ell=1,\dots, n\!-\!1$ and $x\in \Sigma_n$, define
	\begin{equation}\label{vn}
	v_\ell^n(x)=\sqrt{\tfrac{2}{n}}\,\sin\Big(\tfrac{\pi\ell x}{n}\Big)\,,\quad\mbox{and}\quad  \lambda_\ell^n=4n^2\sin^2\Big(\frac{\pi\ell}{2n}\Big)\,. \end{equation}
The functions $\{v_\ell^n\,;\;\ell=1,\dots, n\!-\!1\}$ are the eigenfunctions  and $\{-\lambda_\ell^n;\; \ell=1,\dots, n\!-\!1\}$ are the eigenvalues of the operator $n^2\mathfrak  B_n^0 $. Moreover, $\{v_\ell^n\,;\;\ell=1,\dots, n\!-\!1\}$ is an ortonormal basis of $\mc D(\mathfrak  B_n^0)$. Thus, we can express $ P_{t}^{n,0}(x,y)$ in terms of this basis as 
\begin{equation*}
P_{t}^{n,0}(x,y)=\sum_{\ell=1}^{n-1}e^{-\lambda_\ell^n t}\,	v_\ell^n(x)\,	v_\ell^n(y).
\end{equation*}
Using last  expression  and integrating twice on time, we get
\begin{equation}\label{int}
\int_0^t\int_0^rP_{r-s}^{n,0}(x,1)\,ds\,dr= \sum_{\ell=1}^{n\!-\!1}t^2\,\psi(\lambda_\ell^n t)	v_\ell^n(x)\,v_\ell^n(1),
\end{equation}
where  $$\psi(u):=\frac{e^{-u}-1+u}{u^2}.$$ Note that $|\psi(u)|\lesssim \min \{1,\frac{1}{u}\}$, for all $u\geq 0$. Recall that we need to consider $x=1$ and $x=n-1$.  First we  analyse the case $x=1$, so that in \eqref{int} we have 
\begin{equation*}%\label{int_1}
\sum_{\ell=1}^{n\!-\!1}t^2\,\psi(\lambda_\ell^n t)\,	(v_\ell^n(1))^2\leq t\,\sum_{\ell=1}^{n\!-\!1}\frac{(v_\ell^n(1))^2}{\lambda_\ell^n}.
\end{equation*}
Using  \eqref{vn}, the Double-angle formula for sine and the Half-angle formula for cosine, we have 
\begin{equation*}%\label{int_2}
\sum_{\ell=1}^{n-1}\frac{(v_\ell^n(1))^2}{\lambda_\ell^n}= \frac{1}{n^3}\Big[n-1+\sum_{\ell=1}^{n-1} \cos\big( \ell \pfrac{\pi}{n}\big)\Big].
\end{equation*}
We claim that
\begin{equation*}%\label{sum_cos}
\sum_{\ell=1}^{n-1} \cos\big( \ell \pfrac{\pi}{n}\big)=0\,,
\end{equation*}
which ends  the proof of \eqref{int_estimate}. The claim follows from the general identity
\begin{equation*}
\sum_{\ell=1}^{n-1} \cos\big( \ell \theta\big)=\frac{\cos\big(\theta\,\pfrac{n}{2}\big)\,\sin\big(\theta\,\pfrac{n-1}{2}\big)}{\sin\big(\pfrac{\theta}{2}\big)},
\end{equation*}
taking $\theta =  \pfrac{\pi}{n}$. To prove this identity we denote 
$$S= 
\sum_{\ell=1}^{n-1} \cos\big( \ell \theta\big)~~~\mbox{ and }~~~
S'= 
\sum_{\ell=1}^{n-1} e^{i\ell \theta}.$$
Since $S$ is equal to the real part of $S'$, we will obtain an expression for $S'$ and then take the real part of it to get the value for $S$. Using the formula for the  finite geometric series  for $S'$, we get
$$S'=\frac{e^{i\theta}\,( 1-e^{i\theta(n-1)})}{1-e^{i\theta}}.$$
Doing  some computations it is easy to see that  $1-e^{i\alpha}=2i\,\sin\big(\pfrac{\alpha}{2}\big)\,e^{i\alpha/2}$, for any angle $\alpha$, so that
$$S'\,=\;\;e^{i\,\theta\,\pfrac{n}{2}}
\frac{\sin\big(\theta\,\pfrac{n-1}{2}\big)}{\sin\big(\pfrac{\theta}{2}\big)}.$$

Now, we analyse \eqref{int} for $x=n-1$. Since $v_\ell^n(n-1)\,v_\ell^n(1)=-\cos(\pi\ell)\,(v_\ell^n(1))^2$, we have 
\begin{equation*}
\begin{split}
\int_0^t\int_0^rP_{r-s}^{n,0}(n-1,1)\,ds\,dr\,
&\leq \,\sum_{\ell=1}^{n\!-\!1}t^2\,|\psi(\lambda_\ell^n t)|\,	(v_\ell^n(1))^2,\\
\end{split}
\end{equation*}
and the proof follows as in the case $x=1$.
\end{proof}

\section{Tightness} \label{sec:tightness}

In this section we prove that the sequence of processes $\{\mc Y_t^n; t \in [0,T]\}_{n \in \bb N}$ is tight by using Mitoma's criterion    \cite{Mitoma}. We note that as in \cite{FGN_Robin} we can show that the space $\mathcal S_\theta$ endowed with the semi-norms given in \eqref{semi-norm}
is a Fr\'echet space. Under this criterion we are left to check tightness for the real-valued processes $\{\mathcal Y_t^n(f); t \in [0,T]\}_{n \in \bb N}$ for any $f \in  \mc S_\theta$. By \eqref{int part of mart}, it is enough to show tightness for each term in that martingale decomposition.  We will make use of Aldous' criterion:
\begin{proposition}
 A sequence $\{x_t; t\in [0,T]\}_{n \in \bb N}$ of real-valued processes is tight with respect to the Skorohod topology of $\mc
D([0,T],\bb R)$ if:
\begin{itemize}
\item[i)]
$\displaystyle\lim_{A\rightarrow{+\infty}}\;\limsup_{n\rightarrow{+\infty}}\;\mathbb{P}_{\mu_n}\Big(\sup_{0\leq{t}\leq{T}}|x_{t
} |>A\Big)\;=\;0\,,$

\item[ii)] for any $\varepsilon >0\,,$
 $\displaystyle\lim_{\delta \to 0} \;\limsup_{n \to {+\infty}} \;\sup_{\lambda \leq \delta} \;\sup_{\tau \in \mc T_T}\;
\mathbb{P}_{\mu_n}(|
x_{\tau+\lambda}- x_{\tau}| >\varepsilon)\; =\;0\,,$
\end{itemize}
where $\mc T_T$ is the set of stopping times bounded by $T$.
\end{proposition}
For the proof of tightness of integral terms we will make use of  the so-called {\em Kolmogorov-Centsov criterion}:

\begin{proposition}[Kolmogorov-Centsov's criterion]
\label{prop:KolCen}
A sequence $\{{\bf{X}}_t^n; t \in [0,T]\}_{n \in \bb N}$ of continuous, real-valued, stochastic processes is tight with respect to the uniform topology of $\mc C([0,T]; \bb R)$ if the sequence of real-valued random variables $\{{\bf{X}}_0^n\}_{n \in \bb N}$ is tight and there are positive constants $K,\gamma_1,\gamma_2$ such that
\[
E[|{\bf{X}}_t^n-{\bf{X}}_s^n|^{\gamma_1}] \leq K|t-s|^{1+\gamma_2}
\]
for any $s,t \in [0,T]$ and any $n \in \bb N$.
\end{proposition} 

\subsection{Tightness at the initial time}
The sequence $\{\mc Y_0^n(f)\}_{n \in
\bb N}$ is tight, as a consequence of 
\begin{equation*}
\begin{split}
\bb E_{\mu_n}\Big[\Big(\mathcal Y_0^n(f)\Big)^2\Big]&\;=\;\frac{1}{n}\sum_{x=1}^{n-1}f^2\Big(\tfrac xn\Big)\chi(\rho^n_0(x))+\frac 2n\sum_{x<y}f\Big(\tfrac xn\Big)f\Big(\tfrac yn\Big)\varphi^n_0(x,y)
\end{split}
\end{equation*}
and by 
assumption \eqref{eq:ass_ini_corr}  last expression is  bounded for any value of $\theta\geq 0$.

\subsection{Tightness of integral terms }

Let us now prove tightness for each one of the additive functionals that appear in \eqref{int part of mart}.  We start by showing tightness of the additive functional for the first term at the right hand-side of \eqref{int part of mart}, namely for $$\int_0^t\mc Y_s^n(\Delta_n f)ds.$$

We starting by checking item i) in   Aldous' criterion. By the Tchebychev's inequality and by the Cauchy-Schwarz inequality is is enough to note that
\begin{equation}\label{estimate}
\begin{split}
\mathbb{E}_{\mu_n}\Big[\sup_{t\leq {T}}&\Big(\int_{0}^t\mc Y_s^n(\Delta_n f)\, ds\Big)^2\Big]\;\leq\; T \int_{0}^T \mathbb{E}_{\mu_n}\Big[\Big(\frac{1}{\sqrt{n}}\sum_{x=1}^{n-1}\Delta_n f(\tfrac{x}{n})(\eta_{sn^2}(x)-\rho^n_s(x))\Big)^2\Big]\, ds\\
\leq&\frac{T^2}{{n}}\sum_{x=1}^{n-1}\big(\Delta_n f(\tfrac{x}{n})\Big)^2\sup_{t\leq{T}}\chi(\rho^n_t(x))
+\frac{T^2}{{n}}\sum_{\at{x,y=1}{x\neq y}}^{n-1}\Delta_n f(\tfrac{x}{n})\Delta_n f(\tfrac{y}{n})\sup_{t\leq{T}}\varphi^n_t(x,y).
\end{split}
\end{equation}
From Proposition \ref{prop:corr_bound} and since $f\in\mc S_\theta$, last expression is bounded from above by a constant. 

To check item ii) of Aldou's criterion, we  use the same argument  as in item i).  
We take a stopping time $\tau \in \mc T_T$,  we apply  Tchebychev's inequality together  with \eqref{estimate}, to get that
 \begin{equation*}
\mathbb{P}_{\mu_n}\Big(\Big|  \int_{\tau}^{\tau+\lambda}\mc Y_s^n(\Delta_n f)\, ds\;\Big| >\varepsilon\Big)
	\;\leq\; \frac{1}{\varepsilon^2} \mathbb{E}_{\mu_n}\Big[ \Big(  \int_{\tau}^{\tau+\lambda}\mc Y_s^n(\Delta_n f)\; ds \;\Big)^2\Big]
	\;\lesssim\; \frac{\delta^2 }{\varepsilon^2}\,,
\end{equation*}
which vanishes as $\delta\rightarrow{0}$.

Now we prove tightness for the remaining additive functionals that appear at the right hand-side of \eqref{int part of mart}. In this case the Aldou's criterium  is not sufficient to prove tightness for those terms. The main problem is that all the terms have a factor of $n$ in front of them and the bounds that we have when we apply the Cauchy-Schwarz inequality are not good enough to kill those factors of $n$. What we do instead is that we apply Kolmogorov-Centsov's criterion stated in Proposition \ref{prop:KolCen}. We do the proof for one of the terms but we note that for the others it is completely analogous. 

We prove now tightness for terms of the form
$${\bf{X}}_t^n=\int_0^t C_n^\theta  (\eta_{sn^2}(x)-\rho_{s}^n(x))ds\,,$$
where $x=1$ or $x=n-1$. From \eqref{int part of mart} and since $f\in\mc S_\theta$, we see that  above we need to take $C_n^\theta=\sqrt n \textbf{1}_{\{\theta<1\}}+n^{3/2-\theta}\textbf{1}_{\{\theta>1\}}.$ 
We will prove tightness of last term by estimating the $\bb L^2(\bb P_{\mu_n})$-norm of ${\bf{X}}_t^n-{\bf{X}}_s^n$ so that we will take $\gamma_1=2$ in Proposition \ref{prop:KolCen}.  The proof is similar to the one of Lemma \ref{lem:conv_Dyn_mart} so that we omit some computations. 
By developing the square in the expectation  we have  that
\begin{equation*}
\begin{split}
\bb E_{\mu_n} \Big[\Big(\int_s^t C^\theta_n&(\eta_{rn^2}(x)-\rho_{r}^n(x))\, dr\Big)^2\Big]=2(C_n^\theta)^2\int_s^t\int_s^r\phi_{r,u}^n(1,1)\,\, du\, dr\\
=\;&2\;(C_n^\theta)^2\int_s^t\int_s^r\Big\{\sum_{z\neq 1}P_{r-u}^{n,\theta}(1,z)\varphi_{u}^n(1,z)+P_{r-u}^{n,\theta}(1,1)\chi(\rho^n_{u}(1))\Big\}\, du\, dr\,.
\end{split}
\end{equation*}
Now note that since $\sum_{z\neq 1}P_{r-s}^{n,\theta}(1,z)$ is bounded by one, uniformly on time and on $n$ and from  Proposition  \ref{prop:corr_bound} we can conclude that
\begin{equation*}
\begin{split}
(C_n^\theta)^2\int_s^t\int_s^r\sum_{z\neq 1}P_{r-u}^{n,\theta}(1,z)\varphi_{u}^n(1,z)\, du\, dr\lesssim (t-s)^2\,.
\end{split}
\end{equation*}

Now we analyse the remaining term and it is here that we need an extra argument with respect to the proof of Lemma \ref{lem:conv_Dyn_mart}. By looking at \eqref{bound_2} we see that the bound is of order $t$. For Kolmogorov-Centsov's criterion, this bound is not enough, we need to obtain an exponent a bit  bigger that one. For that purpose we note that since $\chi(\rho_s^n(1))$  is bounded by one, uniformly on time and on $n$  and from \eqref{bound_coupling} the proof of tightness ends as long as we show, for $x=1$ and for $x=n-1$ that
\begin{equation}\label{estimate_delta}
(C_n^\theta)^2n^\theta\int_s^t\int_s^rP_{r-u}^{n,0}(x,1)\, du\, dr \lesssim (t-s)^{1+\delta_\theta},
\end{equation}
where $\delta_\theta=|\tfrac{1-\theta}{2}|{\textbf{1}}_{\theta<3}+{\textbf{1}}_{\theta\geq 3}$. 
To prove the previous estimate in the case $\theta\geq 3$ we just observe that
 $(C_n^\theta)^2n^\theta=n^{3-\theta}\leq 1$, then 
 \begin{equation*}
(C_n^\theta)^2n^\theta\int_s^t\int_s^rP_{r-u}^{n,0}(x,1)\, du\, dr \lesssim (t-s)^{2}\,.
\end{equation*}
For the case $\theta<3$,
we repeat the computations of the proof of Lemma \ref{int_estimate_lemma} so that many steps are sketched. We start with the case $x=1$, but we note that $x=n-1$ is completely analogous. As in \eqref{int}, the  time  integral at the left hand-side of last  display can be written as 
\begin{equation}\label{int_sum}
\int_s^t\int_s^rP_{r-u}^{n,0}(x,1)\,du\,dr\,=\, \sum_{\ell=1}^{n\!-\!1}(t-s)^2\,\psi(\lambda_\ell^n (t-s))(v_\ell^n(1))^2.
\end{equation}
To handle with the sum above we observe
 that  $\psi(u)\leq\frac{1}{u}$, for $u>0$. Plugging this estimate in the expression above we have that
\begin{equation}\label{int_int_sum}
(C_n^\theta)^2n^\theta\int_s^t\int_s^rP_{r-u}^{n,0}(x,1)\, du\, dr \leq (C_n^\theta)^2n^\theta (t-s)\sum_{\ell=1}^{n\!-\!1} \frac{(v_\ell^n(1))^2}{\lambda_\ell^n}\,.
\end{equation}
Let $\delta\geq 0$  and  $t-s \geq \frac{1}{\lambda_\ell^n}$, we  rewrite the expression above as
\begin{equation*}
 (C_n^\theta)^2n^\theta (t-s)\sum_{\ell=1}^{n\!-\!1} \frac{(v_\ell^n(1))^2}{\lambda_\ell^n}\,(\lambda_\ell^n)^{\delta}\Big(\frac{1}{\lambda_\ell^n}\Big)^\delta\leq\,(C_n^\theta)^2n^\theta (t-s)^{1+\delta}\sum_{\ell=1}^{n\!-\!1} \frac{(v_\ell^n(1))^2}{(\lambda_\ell^n)^{1-\delta}}\,.
\end{equation*}
By the expressions of the eigenfunction $v_\ell^n$ and the eigenvalues $\lambda_n^\ell$, see \eqref{vn}, and 
using the Double-angle formula for sine, we can bound from above the right hand-side of last display  by
\begin{equation*}
\frac{(C_n^\theta)^2n^\theta}{n^{3-2\delta}} (t-s)^{1+\delta}\sum_{\ell=1}^{n\!-\!1}\sin^{2\delta}(\tfrac{\pi\ell}{2n})\cos^2(\tfrac{\pi\ell}{2n})\lesssim \frac{(C_n^\theta)^2n^\theta}{n^{2-2\delta}} (t-s)^{1+\delta}.
\end{equation*}
Thus, if $\theta<1$, since $C_n^\theta=\sqrt n$, we have that the right hand-side of last display is equal to $(t-s)^{1+\delta}$ for the choice $\delta_\theta=(1-\theta)/2$, while for $1<\theta<3$, since $C_n^\theta= n^{3/2-\theta}$, the right hand-side of last display is equal to $(t-s)^{1+\delta}$ for the choice $\delta_\theta=(\theta-1)/2$. Note that for this choice $\delta_\theta\in[0,1)$.  Although this information is not relevant when $t-s\geq  \frac{1}{\lambda_\ell^n}$,  in the case  $t-s< \frac{1}{\lambda_\ell^n}$ it is totally necessary, because in this case it will appear $1-\delta$ in the exponent and it must be positive, in order to get the correct bound, see \eqref{1-delta}. 
To handle with the case $t-s< \frac{1}{\lambda_\ell^n}$ we start by observing  that $\psi(u)\leq e$, for $0<u\leq 1$. Then, using
\eqref{int_sum}, we have
\begin{equation}\label{from_56}
(C_n^\theta)^2n^\theta\int_s^t\int_s^rP_{r-u}^{n,0}(x,1)\, du\, dr \lesssim\,(C_n^\theta)^2n^\theta (t-s)^2\,\sum_{\ell=1}^{n\!-\!1}(v_\ell^n(1))^2\,.
\end{equation}
Rewriting the expression above, using that $1-\delta>0$ and 
recalling that $t-s<\frac{1}{\lambda_\ell^n}$, we have
\begin{equation}\label{1-delta}
(C_n^\theta)^2n^\theta\,(t-s)^{1+\delta}\;\sum_{\ell=1}^{n\!-\!1}(v_\ell^n(1))^2\,(t-s)^{1-\delta}\,\leq \,
(C_n^\theta)^2n^\theta\,(t-s)^{1+\delta}\,\sum_{\ell=1}^{n\!-\!1}\frac{(v_\ell^n(1))^2}{(\lambda_\ell^n)^{1-\delta}},
\end{equation}
and the proof follows as above. Note that the choice of $\delta_\theta$ is the same, that is, $\delta_\theta=|\frac{\theta-1}{2}|$, for $\theta<3$ 
and  the proof ends.

\subsection{Tightness of martingales}
 We know from Lemma \ref{lem:conv_mart} that  the sequence of martingales
converges, and,  in particular, it is tight.

\section{Proof of Proposition \ref{prop:corr_bound}}\label{prova_corr_bound}
\label{s6}
We split the proof of this proposition in two settings: first we treat the case $\theta<1$ and then we treat the case $\theta>1$. The main difference between the two regimes is that for $\theta<1$ we use a comparison with a two-dimensional RW which has slow rates at the boundary of $V_n$, while for $\theta>1$ we make a comparison with a two-dimensional RW which is reflected at the lines $x=1$ and $y=n-1$. From here on we do not impose any condition on $\theta$ but at some point we will see that we will need to consider $\theta<1$. 
The steps in the proof of Proposition \ref{prop:corr_bound} are : first, recall that the correlation function is solution to  the discrete equation \eqref{eq:disc_eq_corr}; second, use Duhamel's formula to write the correlation function in terms of a two dimensional random walk; finally, prove bounds on the transition probabities of those random walks.

 Recall that $\varphi^n_t(\cdot,\cdot)$ is solution of \eqref{eq:disc_eq_corr}
 and recall that $\{\mc X_{tn^2}^\theta; \,t\geq 0\}$ is the RW with generator $n^2\A_n^\theta$ which is absorbed in $\p V_n$. 
Denote by $\mathcal{P}_{u}$ and $\mathcal{E}_{u}$   the corresponding probability and expectation, respectively, starting from the position $u\in V_n\cup \p V_n$.
A simple computation, as done  in Subsection 8.1 of \cite{FGN_Robin}, shows that
\begin{equation}\label{eq17}
\varphi_t^n(x,y)\;=\; {\mc E}_{(x,y)}\Big[\varphi^n_0(\mc  X_{tn^2}^\theta)+ \int_0^t g^n_{t-s}( \mc X_{sn^2}^\theta)\,ds\Big]\,.
\end{equation}  
The function $ g^n_t$ defined in the last display was introduced in \eqref{g}. The tools to prove last identity are: ${\mc E}_{(x,y)}[f( \mc X_{tn^2}^\theta)]=(e^{tn^2\A_n^\theta}f)(x,y)$ is a semi-group, Kolmogorov's forward equations and Leibniz Integral Rule.
Then 
\begin{equation}\label{eq17a}
\max_{(x,y)\in V_n}|\varphi_t^n(x,y)|\leq \max_{(x,y)\in V_n}|\varphi^n_0(x,y)|+ \max_{(x,y)\in V_n}\Big|{\mathcal  E}_{(x,y)}\Big[\int_0^t g^n_{t-s}( \mc X_{sn^2}^\theta)\,ds\Big]\Big|\,.
\end{equation}
Due to \eqref{eq:ass_ini_corr_1} and  \eqref{eq:ass_ini_corr}, in order to finish the proof,  it remains to deal with the second term on the right hand side of  last expression. Note that since
the operator $n^2\A_n^\theta$ is a bounded operator (for $n$ fixed) it generates  an uniformly continuous semigroup $\{e^{sn^2\A_n^\theta};\,s\geq 0\}$ on $ V_n\cup\p V_n$.
By Fubini's Theorem
\begin{equation}\label{eq20a}
{\mc E}_{(x,y)}\Big[\int_0^t g^n_{t-s}( \mc X_{sn^2}^\theta)\,ds\Big]= \int_0^t \big(e^{sn^2\A_n^\theta}g^n_{t-s}\big)(x,y)\,ds\,.
\end{equation}
Changing variables, the right hand side of \eqref{eq20a} can be written as 
$$\int_0^t \big(e^{(t-r)n^2\A_n^\theta}g^n_r\big)(x,y)\,dr \,.$$
Now the proof, in the case $\theta<1$, ends as a consequence of the  next two lemmas.
\begin{lemma}\label{lema_da_prop_2.1}
	We have that
	\begin{equation*}
	\sup_{t\geq 0}\max_{(x,y)\in V_n}\Big|\int_0^t \big(e^{(t-r)n^2\A_n^\theta}g^n_r\big)(x,y)\,dr\Big|\;\lesssim \frac{n+n^\theta}{n^2},
	\end{equation*}
	\begin{equation*}
	\sup_{t\geq 0}\max_{(x,y)\in V_n}\Big|\int_0^t \big(e^{(t-r)n^2\A_n^\theta}g^n_r\big)(x,y)\,dr\Big|\;\lesssim \frac{n^\theta}{n^2},\quad\textrm{for} \, \, x=1,n-1.
	\end{equation*}
		
\end{lemma}
\begin{proof}
	Since the function $g^n_r$ defined in \eqref{g} is supported on the diagonal 
	\begin{equation}\label{diagonal}
	\mc D_n=\{(z,z+1)\,;\; z=1,\dots, n-2\}\,,
	\end{equation}we can rewrite  $(e^{(t-r)n^2\A_n^\theta}g^n_r)(x,y)$ as 
	\begin{equation*}
	\sum_{z=1}^{n-2}e^{(t-r)n^2\A_n^\theta}\big((x,y),\,(z,z+1)\big)\,g^n_r(z,z+1)\,.
	\end{equation*}
	Then, for all $(x,y)\in V_n$,
	\begin{equation}\label{eq22}
	\Big|\int_0^t \big(e^{(t-r)n^2\A_n^\theta}g^n_r\big)(x,y)\,dr\Big|\;\leq \; S_n\cdot\int_0^t \sum_{z=1}^{n-2}e^{(t-r)n^2\A_n^\theta}\big((x,y),\,(z,z+1)\big)\,dr\,,
	\end{equation}
	where 
	\begin{equation}\label{S_n}
	S_n\;=\;\;\sup_{ r\geq 0}\max_{z\in\{1,\dots,n-2\}}|g^n_r(z,z+1)|\,.
	\end{equation}
	First we will work with the time integral on the right hand side of \eqref{eq22}. By  the equality $$(e^{(t-r)n^2\A_n^\theta})(u,v)=\mc {P}_{u}[\mc X_{s}^\theta=v],$$ together with  a change of   variables and the  definition of  $\mc D_n$, we get
	\begin{equation*}
	\int_0^t \sum_{z=1}^{n-2}e^{(t-r)n^2\A_n^\theta}\big((x,y),\,(z,z+1)\big)\,dr
	%\int_0^{tn^2}\sum_{z=1}^{n-2}{\mc P}_{(x,y)}\big[X_{s}=(z,z+1)\big]\,\frac{ds}{n^2}
	\,=\,
	\int_0^{tn^2}\mc {P}_{(x,y)}\big[\mc X_{s}^\theta\in \mc D_n\big]\,\frac{ds}{n^2}.
	\end{equation*}
	Extending the interval of integration to infinity and applying Fubini's theorem on  the last integral, we bound it from above by 
	\begin{equation*}
	\frac{1}{n^2}\,{ \mc E}_{(x,y)}\Big[\int_0^{\infty}\Ind{\mc X_s^\theta \in \mc D_n}\,ds\Big]\,.
	\end{equation*}
	Note that the expectation above is  the total time spent by the RW $\{ \mc X_s^\theta;\,s\geq 0\}$ on the diagonal $\mc D_n$.
	By Section \ref{2d_coupling_section}, we have
	\begin{equation}\label{2d_rw_diagonal}
	{\mc E}_{(x,y)}\Big[\int_0^{\infty}\Ind{\mc X_s^\theta \in \mc D_n}\,ds\Big]\leq x \tfrac{n-y}{n-1}+n^{\theta}\,.
	\end{equation} 
	The term $x\frac{n-y}{n-1}$ is the improvement of this proof over the one in \cite{bmns}. Note that for the choice $x=1$ the last bound is $O(n^\theta).$ 
	Thus,  the integral on the right hand side of \eqref{eq22} is $O\Big(\frac{n^{\theta}}{n^2}\Big)$ and for $x\neq 1$ it is $O\Big(\frac{n+n^\theta}{n^2}\Big)$. This ends the proof.
\end{proof}
Since we have the estimates given  in Lemma \ref{lema_da_prop_2.1}, in order
  to conclude the proof of Proposition \ref{prop:corr_bound} for the case $\theta<1$, we need to bound  $S_n$ (which was defined in \eqref{S_n}) by a constant. This is the content of the next lemma. Now, we note that  the estimate obtained in the  Lemma \ref{lema_da_prop_2.1}  is good for our purposes only in the case $\theta<1$. When $\theta>1$ we need to redo the proof of Proposition \ref{prop:corr_bound}.   The idea is to rewrite \eqref{eq:disc_eq_corr} in terms of the generator of the bi-dimensional RW $\{\mathcalboondox X_{tn^2}; \,t\geq 0\}$ which is  reflected at the lines $x=1$, $y=n-1$ and at the diagonal $\mc D_n$. A simple computation shows that $\varphi^n_{t}$ is a solution of
\begin{equation}\label{eq:disc_eq_corr_ref}
\begin{cases}
\partial_t \varphi_t^n(x,y)=n^2 \mathcalboondox{ R}^2_n \varphi^n_t(x,y) +g_t^n(x,y)+\mc V(t,x,y)\varphi_t^n(x,y), & \textrm{ for } (x,y)\in V_n,\; t>0,\\
\varphi_t^n(x,y)=0, & \textrm{ for } (x,y)\in \partial  V_n, \;t>0,\\
\varphi_0^n(x,y)=\bb E_{\mu_n}[\eta_0(x)\eta_0(y)]-\rho_0^n(x)\rho_0^n(y), & \textrm{ for } (x,y)\in V_n\cup \partial V_n,\\
\end{cases}
\end{equation}
where  $g^n_{t}(\cdot,\cdot)$ is given in \eqref{g}, $\mc V(t,x,y)=-\textbf{1}_{x=1} n^{2-\theta}-\textbf{1}_{y=n-1}n^{2-\theta}$ and $\mathcalboondox{ R}^2_n$ is the generator of the bi-dimensional  RW, which is reflected at the lines $x=1$,  $y=n-1$ and at the diagonal $\mc D_n$ and acts on  $f:V_n\cup \partial V_n \to\bb R$ as
 \begin{equation}\label{2dRRW}
n^2 (\mathcalboondox{ R}^2_nf)(x,y)=n^2(f(x+1,y)+f(x-1,y)+f(x,y+1)+f(x,y-1)-4f(x,y)),
 \end{equation}
 for $x,y\in V_n$ and $x\neq 1$ and $y\neq n-1$. At the diagonal we have
  \begin{equation*}
n^2 (\mathcalboondox{ R}^2_nf)(x,x+1)=n^2(f(x-1,x+1)+f(x,x+2)-2f(x,x+1))\,,
 \end{equation*}
for $x\in \{2,\dots, n-3\}$, and 
\begin{equation*}
\begin{split}
n^2(\mathcalboondox{R}^2_nf)(1,2)&=n^2(f(1,3)-f(1,2)),\\
n^2 (\mathcalboondox{R}^2_nf)(1,y)&=n^2(f(2,y)+f(1,y+1)+f(1,y-1)-3f(1,y)),\\&\qquad\qquad\qquad\qquad\qquad\qquad\qquad\qquad\qquad\;\;\mbox{ for }y\in\{3,\dots, n-2\},\\
n^2 (\mathcalboondox{R}^2_nf)(1,n-1)&=n^2(f(1,n-2)+f(2,n-1)-2f(1,n-1)),\\
n^2  (\mathcalboondox{R}^2_nf)(x,n-1)&=n^2(f(x-1,n-1)+f(x+1,n-1)+f(x,n-2)-3f(x,n-1)),\\&\qquad\qquad\qquad\qquad\qquad\qquad\qquad\qquad\qquad\;\;\mbox{ for }x\in\{2,\dots, n-3\},\\
n^2  (\mathcalboondox{R}^2_nf)(n-2,n-1)&=n^2(f(n-3,n-1)-f(n-2,n-1)).\\
 \end{split}
 \end{equation*}

By Feynmann-Kac's formula, we have that
\begin{equation*}
\varphi_t^n(x,y)\;=\; {\mathcalboondox E}_{(x,y)}\Big[\varphi^n_0(\mathcalboondox{X}_{tn^2})\,\,e^{\int_0^t\mc V(r,\mathcalboondox{ X}_{rn^2})\,dr}+\int_{0}^t g_{t-s}^n(\mathcalboondox{ X}_{sn^2})\,e^{\int_0^s\mc V(r,\mathcalboondox{ X}_{rn^2})\,dr}\,ds\Big]\,,
\end{equation*}
where  $\{\mathcalboondox{ X}_{tn^2}; \,t\geq 0\}$ is the  RW with generator $n^2\mathcalboondox{ R}_n^2$. Denote by ${\mathcalboondox{P}}_{u}$ and ${\mathcalboondox{ E}}_{u}$   the corresponding probability and expectation, starting from the position $u\in V_n$.
Now, since the function $\mc V$ is negative and repeating  the same arguments as in the proof in the case $\theta<1$ it is enough to note that  the term at the right hand-side of last display  is bounded from above by 
	\begin{equation*}
	S_n\,\cdot\,\int_0^{t}\mathcalboondox{P}_{(x,y)}(\mathcalboondox{ X}_{sn^2} \in \mc D_n) \,ds\,,
	\end{equation*}
	where $S_n$ and $\mc D_n$ were defined in \eqref{S_n} and \eqref{diagonal}, respectively.
	Note that the probability  above is  the  probability that  the RW $\{ \mathcalboondox{ X}_{sn^2};\,s\geq 0\}$ reaches the diagonal $\mc D_n$ starting from $(x,y)$.	 From  Lemma  \ref{lem:discrete_gradient}, the proof ends as long as we show that  the previous integral  is of order $O(\tfrac 1n)$, which is done in  Subsection \ref{2d_reflected_rw}.
	 
%	 .   To prove it, the idea is to use a similar argument as the one used in Section \ref{1d_reflected_rw}, in which we compare the bi-dimensional reflected RW in $V_n$ with the bi-dimensional symmetric simple RW in $\bb Z^2$.  It is done in  Subsection \ref{2d_reflected_rw}. More precisely,  we prove in \eqref{eq:final_estimate_2d} that  the previous display is of order $O(\tfrac 1n)$ as desired.

\begin{lemma}\label{lem:discrete_gradient}Let $\rho^n_t(\cdot)$ be the solution of  \eqref{eq:disc_heat}. Then, its discrete derivative satisfies: 
	\begin{equation}\label{eq23}
	\big|\rho^n_t(x+1)-\rho^n_t(x)\big|\lesssim \frac{1}{n}\,,
	\end{equation}
	for all $x\in\{1,\dots,n-2\}$ and uniformly in $t\geq 0$, for all $\theta\geq 0$.
\end{lemma}

\begin{proof}
We want to find a function  $\phi$, such that $\phi(t,\tfrac xn)$ is close to $\rho_t^n(x)$ and which has some regularity. More precisely, we will consider a sequence of functions which are of class  $C^4$ in space, in such a way that the error between their discrete laplacian and their continuous laplacian is of order $O(\tfrac {1}{n^2})$. To have such functions, it is here where we need to restrict to  initial profiles $\rho_0$ of class $C^6$, see the assumption above \eqref{assumption 1}.  We are going to consider the following sequence of functions $\{\phi_n(t,u)\}_{n\geq 1}$ where 
 $\phi_n(t,u)$ is the solution of 
 \begin{equation}\label{eq:robin}
\begin{cases}
\p_t \phi_n(t,u)\;=\; \p_u^2 \phi_n(t,u)\,, & \textrm{ for } t>0\,,\, u\in (0,1)\,,\\
\p_u \phi_n(t,0) \;=\;\mu_n( \phi_n(t,0)-\alpha)\,, & \textrm{ for } t>0\,,\\
\p_u \phi_n(t,1) \;=\;\mu_n( \beta-\phi_n(t,1))\,, & \textrm{ for } t>0\,,\\
%\phi_n(t,0)\;=\;\alpha\,,\;\; \phi_n(t,1)=\beta\,,& \textrm{ for } t>0\,,\\
\phi_n(0,u)\;=\;\rho_0(u)\,,& u\in [0,1]\,,
\end{cases}
\end{equation}
where  for $\theta<1$ we take $\mu_n=\tfrac {n}{n^\theta-1}$ and for $\theta>1$ we take $\mu_n=\tfrac {n}{n^\theta}$. Note that for $\theta>1$ we have that $\mu_n\to 0.$
In Subsection \ref{ap:heat_robin}  we prove that if $\rho_0\in C^6$ then $\phi_n\in C^{1,4}$. Now if  $\gamma^n_t(x):=\rho^n_t(x)-\phi_n(t,\tfrac xn)\,$  for $x\in{\Sigma_n}\cup\{0,n\}$, then $\gamma^n_t$ is solution of 
\begin{equation}\label{eq28}
\left\{
\begin{array}{ll}
 \partial_t\gamma_t^n(x)\,=\,(n^2\mathfrak B^\theta_n \gamma_t^n)(x)+F_t^n(x)\,, \;\; x\in\Sigma_n\,,\;\;t \geq 0\,,\\
 \gamma_t^n(0)=0\,, \quad \gamma^n_t(n)=0\,, \;\;t \geq 0\,,\\
\end{array}
\right.
\end{equation}
where $\mathfrak B^\theta_n$ was defined in \eqref{eq:operator_B_n_theta}, for $x\in\{2,\ldots,n-2\}$, $F_t^n(x)=(n^2\mathfrak{B}^\theta_n-\partial_u^2)\phi_n(t,\tfrac xn)$.
  We note that since $\phi_n(t,\cdot)\in C^4$, the result follows as long as we show that 
$\big|\gamma^n_t(x)\big|\leq \tfrac Cn$.  Note that
\begin{equation*}
\gamma_t^n(x)\;=\; \bb E_x\Big[\gamma^n_0( \mathfrak X_{tn^2}^\theta)+\int_{0}^t F_{t-s}^n( \mathfrak X_{sn^2}^\theta)\,ds\Big],
\end{equation*}
where  we recall that $\{\mathfrak X_s^\theta,\, s\geq 0\}$ is  the RW on $\bar{\Sigma}_n$, with generator $\mathfrak B^\theta_n$, absorbed at the boundary $\{0,n\}$ and  $\bb E_x$ denotes the  expectation with respect to the probability induced by the generator $\mc B^\theta_n$ and the initial position $x$. Note that this RW was already defined in the beginning of Subsection \ref{1d_coupling_section}. Therefore,
\begin{equation*}
\sup_{t\geq 0}\max_{x\in\Sigma_n}|\gamma_t^n(z)|\;\leq\; \max_{x\in \Sigma_n}|\gamma^n_0(x)|\;+\;\sup_{t\geq 0}\max_{x\in \Sigma_n}\Big|\bb E_x\Big[\int_{0}^t F_{t-s}^n(\ \mathfrak X_{sn^2}^\theta)\,ds\Big]\Big|\,.
\end{equation*}
From assumption \eqref{assumption 1} the first term satisfies the required bound. It remains to analyse the term on the right hand side of last display, which  can be written as 
\begin{equation}\label{eq:imp_1}
\int_{0}^{t}\sum_{z=1}^{n-1}\bb P_x\Big[ \mathfrak X_{sn^2}^\theta=z\Big]\cdot F_{t-s}^n(z)\,ds\,.
\end{equation}
Since  $\phi_n$ is of class  $C^4$,  then  $F_t^n(x)\lesssim{1/n^2}$ for any $x\in\{2,\ldots,n-2\}$ and for any $t\geq 0$. Then,  \eqref{eq:imp_1} is bounded from above by
\begin{equation}\label{713}
\frac{C}{n}+\sum_{k\in \{1,n-1\}} \bb E_x\Big[\int_{0}^{\infty}\textbf{1}_{\{\mathfrak X_{sn^2}^\theta=k\}}\, ds\Big]\cdot|F_t^n(k)|\,.
\end{equation}
The previous   expectation is  the average time spent by the RW at the site $k$ until its absorption which is the solution of the elliptic equation  
\begin{equation*}
\left\{
\begin{array}{ll}
-\mathfrak B^\theta_n\psi^n(x)\;=\;C\delta_{x=k}\,, \;\;\forall\; x\in\Sigma_n\,,\\
 \psi^n(0)\;=\;0\,, \quad \psi^n(n)\;=\;0.
\end{array}
\right.
\end{equation*}
Above  $C$ is a constant. 
 When $k=1$ and $k=n-1$, a simple computation shows that $\max_{x=1,\ldots,n-1}|\psi^n(x)|\lesssim \tfrac {n^\theta}{n^2}$. We leave the details of this computation to the reader.  Moreover,  from the boundary conditions  we easily  obtain that $|F_t^n(k)|\lesssim 1$ for any $t\geq 0$ and for $k=1$ and $k=n-1$. This ends the proof for the case $\theta<1$.

 In the case $\theta>1$ the previous bonds are not good. In order to overcome the problem, we rewrite \eqref{eq28} in terms of the generator of the one-dimensional RW which is reflected at the sites $x=1$ and $x=n-1$. 
A simple computation shows that $\gamma^n_t$ is also a solution of 
\begin{equation}\label{eq:refl}
 \partial_t\gamma^n_t(x)\,=\,(n^2\mathcalboondox{ R}_n \gamma^n_t)(x)+V(t,x)\gamma^n_t(x)+G_t^n(x)\,, \;\; x\in\Sigma_n\,,\;\;t \geq 0\,,
\end{equation}
where $V(r,x)=-\textbf{1}_{x=1} n^{2-\theta}-\textbf{1}_{x=n-1}n^{2-\theta}$, 
the operator $\mathcalboondox{ R}_n $  acts on functions $f:{\Sigma}_n\to\bb R$ as
\begin{equation}\label{eq:operator_R_n}
\begin{split}
n^2(\mathcalboondox{ R}_n f)(x)&\;=\;\Delta_nf(x)\,, ~~\textrm{ for } x\in \{2,\cdots, n-2\},\\
n^2(\mathcalboondox{ R}_nf)(1)&=n\nabla^+_nf(1)=n^2(f(2)-f(1))\,,\\n^2(\mathcalboondox{ R}_nf)(n-1)&=n\nabla ^-_nf(n-1)=n^2(f(n-1)-f(n-2))\,
\end{split}
\end{equation}
and the function  $G_t^n(x)=(n^2\mathcal{R}_n-\partial_u^2)\phi_n(t,\tfrac xn)$, for $x=2,...,n-2$,
\begin{equation*}
\begin{split}
&G_t^n(1)=(n^2\mathcalboondox{ R}_n-\partial_u^2)\phi_n(t,\tfrac 1n)+\tfrac {n^2}{n^\theta}(\alpha-\phi(t,\tfrac {1}{n}))\,,\\
&G_t^n(n-1)=(n^2\mathcalboondox{ R}_n-\partial_u^2)\phi_n(t,\tfrac {n-1}{n})+\tfrac {n^2}{n^\theta}(\beta-\phi(t,\tfrac {n-1}{n}))\,.
\end{split}\end{equation*}
As above, the result follows as long as we show that 
$\big|\gamma^n_t(x)\big|\lesssim \tfrac 1n$.  By Feynmann-Kac's formula, we have that
\begin{equation*}
\gamma_t^n(x)\;=\; \mathfrak E_x\Big[\gamma^n_0(\mathfrak X_{tn^2})e^{\int_0^tV(r,\mathfrak X_{rn^2})\,dr}+\int_{0}^t G_{t-s}^n(\mathfrak X_{sn^2})e^{\int_0^sV(r,\mathfrak X_{rn^2})\,dr}\,ds\Big]\,,
\end{equation*}
where $\{\mathfrak X_s,\, s\geq 0\}$ is  the one-dimensional reflected RW on ${\Sigma_n}$, with generator $\mathcalboondox{ R}_n$. Above, $\mathfrak E_x$ denotes the  expectation with respect to the probability induced by the generator $\mathcalboondox{ R}_n$ and the initial position $x$.  Since $V$ is a negative function and  does not depend on time, the term at the left hand-side of last expression can be bounded from above by
\begin{equation*}
\max_{x\in \Sigma_n}|\gamma^n_0(x)|. 
\end{equation*}
Now we bound the remaining term. A simple computation, based on Taylor expansion of the function $\phi_n(t,\cdot)$, shows that  $G_t^n(x)\lesssim{1/n^2}$ for any $x\in\{2,\ldots,n-2\}$ and for any $t\geq 0$ and $|G_t^n(x)|\lesssim  1$ for any $t\geq 0$ and for $x=1$ and $x=n-1$.
Again since $V$ is a negative function we have simply to bound
\begin{equation*}
\sup_{t\geq 0}\max_{x\in \Sigma_n}\Big|\mathfrak E_x\Big[\int_{0}^t G_{t-s}^n(\mathfrak X_{sn^2})\,ds\Big]\Big|\,=\,\sup_{t\geq 0}\max_{x\in \Sigma_n}
\int_{0}^{t}\sum_{z=1}^{n-1}\mathfrak P_x\Big(\mathfrak X_{sn^2}=z\Big)\cdot G_{t-s}^n(z)\,ds\,.
\end{equation*}
From the properties of  $G_t^n(\cdot)$ last term is bounded from above by a constant times
\begin{equation*}%\label{713}
\frac{1}{n}+\sup_{t\geq 0}\max_{x\in \Sigma_n}\sum_{y\in \{1,n-1\}} \int_{0}^{t}\mathfrak P_x\Big(\mathfrak X_{sn^2}=y\Big)\,ds\,.
\end{equation*}
The proof ends now by showing that  last sum is of order $O(\tfrac 1n)$, which  is done in  \eqref{eq:final_estimate_1d}.
\end{proof}

\begin{remark}
We observe that the proof of Proposition  \ref{prop:corr_bound} in the case   $\theta>1$   holds for
any $\theta>0$. We decided to present a different proof for the case  $\theta<1$, because there it
appears the natural RWs associated with this model, as one can see in \eqref{eq:disc_heat} and \eqref{eq:disc_eq_corr} and in \cite{bmns} and \cite{FGN_Robin}.
\end{remark}

\subsection{Time spent on the diagonal by the  bi-dimensional RW}\label{2d_coupling_section}
This subsection is devoted to prove 
 \eqref{2d_rw_diagonal}.
This proof is the same presented in  Section 3 of \cite{bmns}, but here we need a more refined estimate.
Denote the expectation of  the total time spent by the RW $\{ \mc X_s^\theta;\,s\geq 0\}$ on the diagonal $\mc D_n$ by
\begin{equation*}
T_{(x,y)}^\theta:={\mathcal E}_{(x,y)}\Big[\int_0^{\infty}\Ind{\mc X_s^\theta \in \mc D_n}\,ds\Big]\,.
\end{equation*}
By means of a coupling argument we are going to show that  \eqref{2d_rw_diagonal} corresponds to 
\begin{equation}\label{eq:37_reescrita}
T_{(x,y)}^\theta\leq T_{(x,y)}^0+n^{\theta}\,,
\end{equation}
because $T_{(x,y)}^0=x\frac{y-1}{n-1}$. This term $T_{(x,y)}^0=x\frac{y-1}{n-1}$ is the one  that provides the more refined estimate for the expectation above.
 Before presenting the coupling we derive the explicit expression for  $T_{(x,y)}^0$. In order to do this observe that
\begin{equation*}
T_{(x,y)}^\theta
%={\mc  E}_{(x,y)}\Big[\int_0^{\infty}\Ind{\mc X_s^\theta \in \mc D_n}\,ds\Big]
=\int_0^{\infty}{\mc P}_{(x,y)}\Big[\mc X_s^\theta \in \mc D_n\Big]\,ds=
\int_0^{\infty}e^{s\A_n^\theta}\Ind{\mc D_n}(x,y)\,ds\,.
\end{equation*}
Applying the operator $\A_n^\theta$ in the expression above, we  get
\begin{equation*}
\A_n^\theta\, T_{(x,y)}^\theta
%={\mc E}_{(x,y)}\Big[\int_0^{\infty}\Ind{\mc X_s^\theta \in \mc D_n}\,ds\Big]
%=\int_0^{\infty}{\mc P}_{(x,y)}\Big[\mc X_s^\theta \in \mc D_n\Big]\,ds
=
\int_0^{\infty}\A_n^\theta\,e^{s\A_n^\theta}\Ind{\mc D_n}(x,y)\,ds\,.
\end{equation*}
Using Chapman-Kolmogorov equation, we have 
\begin{equation*}
\A_n^\theta\, T_{(x,y)}^\theta
=
\int_0^{\infty}\partial_s\,e^{s\A_n^\theta}\Ind{\mc D_n}(x,y)\,ds=-\Ind{\mc D_n}(x,y)\,.
\end{equation*}
The last equality comes from fundamental theorem of calculus and some properties the  of semigroup.
Then $T_{(x,y)}^\theta$ satisfies
\begin{equation*}
\A_n^\theta\, T_{(x,y)}^\theta
=
-\delta_{y=x+1}\,,\;\mbox{for all }\;\theta\geq 0\,.
\end{equation*}
For $\theta=0$, a simple but long computation, shows that the solution of the equation above is equal to $T_{(x,y)}^0=x \tfrac{n-y}{n-1}$.

The coupling is quite similar to the one presented in  Section 3 of \cite{bmns} and for completeness we recall it here.
The bi-dimensional coupling is the RW $\{\mc Z_s^\theta,s\geq 0\}$ taking values in $V_n\times \bb N$, where $V_n $ was defined in \eqref{V}. The RW  $\mc Z_s^\theta$ starts from $\big((x,y);\,1\big)$ following an independent copy of $\mc X_s^0$ and when $\mc X_s^0$ is absorbed in $\partial V$ the walker $\mc Z_s^\theta$ flips an independent coin (with probability $n^{-\theta}$ of getting heads). If it comes up heads, $\mc Z_s^\theta$ will be absorbed together the copy of $\mc X_s^0$. But,  if it comes up tails the walker $\mc Z_s^\theta$ jumps to the next level and follows another independent copy of $\mc X_s^0$ starting from the same position on $ V_n$ ,where the last copy of  $\mc X_s^0$ was before being absorbed.  An important observation is that the projection of $\mc Z_s^\theta$ on $V_n$ has the same law of $\mc X_s^\theta$.
Denote by $Y$  the geometric random variable  that counts  the number of tails before the first head.
Then,
\begin{equation*}
\begin{split}
T_{(x,y)}^\theta&\leq \;\sum_{k\geq 1} \mc E_{\big((x,y);\,1\big)}\Bigg[\Ind{Y=k-1}\,\cdot\,\sum_{i=1}^k\int_0^{\infty}\Ind{\mc Z_s^\theta \in \mc D_n\times \{i\}}\,ds\Bigg]\\
&=\;\sum_{k\geq 1}\mc P[Y=k-1]\,\sum_{i=1}^k{\mc {E}}_{(x_i,y_i)}\Big[\int_0^{\infty}\Ind{\mc X_s^0 \in \mc D_n}\,ds\Big]\\
&=\;\sum_{k\geq 1}\mc P[Y=k-1]\,\sum_{i=1}^kT_{(x_i,y_i)}^0\,,
\end{split}
\end{equation*}
where $(x_i,y_i)$, for $i=1,\dots,k$, are the points where the RW $\mc Z_s^\theta$ starts on the level $i$. Note that, for $i=2,\dots,k$, the possible points where it happens are of the form $(1,y)$ for $y=2,\dots, n-1$ or $(x,n-1)$ for $x=1,\dots,n-2$.
Since $T_{(x_i,y_i)}^0=x_i\,\frac{n-y_i}{n-1}$, for $i=2,\dots,k$, we have that $T_{(x_i,y_i)}^0\leq 1$. Thus,
\begin{equation*}
\begin{split}
T_{(x,y)}^\theta&\leq \sum_{k\geq 1}\mc P[Y=k-1]\,\big(T_{(x,y)}^0+k-1\big)= T_{(x,y)}^0+E[Y]= x \tfrac{n-y}{n-1}+n^{\theta}.
\end{split}
\end{equation*}

\subsection{The one-dimensional reflected RW }\label{1d_reflected_rw}
The goal of this subsection is to get the bound
\begin{equation}\label{eq:final_estimate_1d}
\sup_{t\geq 0}\max_{x\in \Sigma_n}\sum_{y\in \{1,n-1\}} \int_{0}^{t}\mathfrak P_x\Big(\mathfrak X_{sn^2}=y\Big)\,ds\,\lesssim\frac{1}{n},
\end{equation}
where  $\{\mathfrak X_{tn^2},\;t\geq 0\}$ is  the reflected RW on ${\Sigma_n}$, with generator $\mathcalboondox{ R}_n$, defined in \eqref{eq:operator_R_n}.
The previous bound is used at the end of the proof of Lemma \ref{lem:discrete_gradient}, where we get that the increment of the empirical profile is of order $O(\tfrac 1n)$,  in the case $\theta>1$. This lemma is  important to estimate the coefficient $S_n$ that appears in the proof of  Proposition \ref{prop:corr_bound}.
We start the proof in the case $x=1$. The idea to prove the bound is to write the occupation time of the site $x=1$ 
in terms of the generator of the RW  $\mathfrak X_{sn^2}$ given in \eqref{eq:operator_R_n}. For that purpose, let us take $f(x)=-(n+1-x)^2$ and note that 
\[
  n^2 \mathcalboondox{ R}_n f(x) =
\left\{
\begin{array}{c@{\;;\;}l}
n^2(2n-1) & x=1\\
-2n^2 & x=2,\dots,n-2\\
-13n^2 & x=n-1.\\
\end{array}
\right.
\]
From Dynkin's formula, we know that 
\[
f(\mathfrak X_{tn^2}) -f(\mathfrak X_{0}) - \int_0^t n^2 \mathcalboondox{ R}_n f(\mathfrak X_{sn^2}) ds 
\]
is a mean-zero martingale. By looking at the position where the RW can be at time $sn^2$, we get
\begin{equation*}
\begin{split}
\mathfrak{E}_x\Big[\int_0^t n^2 \mathcalboondox{ R}_n f(\mathfrak X_{sn^2}) ds\Big]&=n^2(2n-1)\int_{0}^{t}\mathfrak P_x\Big(\mathfrak X_{sn^2}=1\Big)\,ds\\&+\sum_{y=2}^{n-1}\int_{0}^{t}\mathfrak P_x\Big(\mathfrak X_{sn^2}=y\Big)n^2\mathcalboondox{ R}_n f(y)ds.\end{split}
\end{equation*}
From last observations we conclude that
\begin{equation*}
\begin{split}
\int_{0}^{t}\mathfrak P_x\Big(\mathfrak X_{sn^2}=1\Big)\,ds\leq\frac{1}{n^2(2n-1)}(13n^2t
+\max_{x,y\in V_n}(f(x)-f(y)).\end{split}
\end{equation*}
Now note that
$$\max_{x,y\in V_n}(f(x)-f(y))=\max_{x,y\in V_n}2(y-x)(n+1-x)+(y-x)^2\leq 3n^2$$ which  ends the proof. To treat the  case $x=n-1$ we repeat exactly the same argument as above but we take instead $f(x)=-x^2$.

\subsection{The bi-dimensional reflected RW}\label{2d_reflected_rw}
In this subsection we prove that
\begin{equation}\label{eq:final_estimate_2d}
\sup_{t\geq 0} \max_{u\in V_n}\int_0^{t}\mathcalboondox{P}_{u}(\mathcalboondox{ X}_{sn^2} \in \mc D_n) \,ds\lesssim \frac{1}{n}.
\end{equation}
The triangle $ V_n $ and its diagonal $ \mc D_n $ were defined in \eqref{V} and \eqref{diagonal} respectively, and $ \mathcalboondox{X}_t $ denotes the continuous time reflected RW on $ V_n $ that jumps to nearest neighbour sites at rate $ 1 $.
 This bound is used at the end of the proof of Proposition \ref{prop:corr_bound}.

Our strategy is exactly the same used in the previous subsection.  For that purpose consider the point  $(x_0,y_0) = (\frac{1}{2},n-\tfrac{1}{2})$ and take $$f(x,y) = -(x-x_0)^2 -(y-y_0)^2.$$ A simple computation shows that

\[
  n^2 \mathcalboondox{ R}^2_n f(x,y) =
\left\{
\begin{array}{c@{\;;\;}l}
n^2(2y_0-5) & x=1, y=2\\
n^2(2x_0-5) & x=1, y\neq 2,n-1\\
n^2(-2n+2x_0+2y_0-2) & x=1, y=n-1\\
n^2(2n-5-2y_0)& x\neq 1,n-2, y=n-1\\
n^2(2n-5-2x_0)& x=n-2, y=n-1\\
n^2(2y_0-2x_0-4)& y=x+1\\
-4n^2 & \textrm{otherwise}
\end{array}
\right.
\]

Using the choice for $(x_0,y_0)$ and repeating the steps of the previous subsection, we conclude that

\begin{equation*}
\begin{split}
\int_0^{t}\mathcalboondox{P}_{u}(\mathcalboondox{ X}_{sn^2} \in \mc D_n) \,ds\leq \frac{1}{n^2(2n-6)}&\Big\{\int_0^t 4n^2\sum_{x,y\in V_n\setminus \mc D_n} \mathcalboondox{P}_{u}(\mathcalboondox{ X}_{sn^2} =(x,y)) ds\\&+ \max_{(x,y),(z,w)\in V_n}f(x,y)-f(z,w)\Big\}.
\end{split}
\end{equation*}

Now note that
\begin{equation*}
\begin{split}
&\max_{(x,y),(z,w)\in V_n}f(x,y)-f(z,w)\\=&\max_{(x,y),(z,w)\in V_n} (z-x)^2+2(z-x)(x-x_0)(w-y)^2+2(w-y)(y-y_0)\\\leq& \,6n^2.\end{split}
\end{equation*}
From the previous computations, in particular, we deduce that
\begin{equation}
\sup_{t\geq 0} \max_{u\in V_n}\int_0^{t}\mathcalboondox{P}_{u}(\mathcalboondox{ X}_{sn^2} \in \mc D_n) \,ds\leq\frac{1}{n^2(2n-6)}\Big(4n^2t+6n^2\Big),
\end{equation}
from where the proof ends.

\subsection{Heat Equation with Robin boundary conditions}
\label{ap:heat_robin}

In this subsection we prove existence of smooth solutions of the heat equation with Robin boundary conditions. 
For that purpose, let $\mu \in (0,\infty)$ and  $\alpha, \beta \in [0,1]$. We consider  the boundary-value problem
\begin{equation}
\label{eq:nonhRobin}
\left\{
\begin{array}{rcll}
\partial_t \rho(t,u) & = & \partial^2_x \rho(t,u)\,,\quad 0<u<1\,,\;\;t>0\,,\\
 \partial_u \rho(t,0)&=&\mu (\rho(t,0)-\alpha)\,,\;\;\;\;t>0\,, \\
 \partial_u \rho(t,1) &=&\mu (\beta-\rho(t,1))\,,\;\;\;\;t>0\,,\\
\rho(0,u) &=& \rho_0(u)\,,\;\;\;\;0\leq u\leq 1\,,
\end{array}
\right.
\end{equation}
where $\rho_0:[0,1]\to[0,1]$ is a measurable profile. 
First we note that
\[
\bar{\rho}(u) : = \tfrac{\beta+ \alpha(1+\mu)}{2 + \mu} + \tfrac{\mu(\beta-\alpha)u}{2+\mu}
\]
is a stationary solution of \eqref{eq:nonhRobin}. If $\tilde{\rho}$ is a solution of \eqref{eq:nonhRobin}, then $\tilde{\rho}(\cdot,\cdot) - \bar \rho(\cdot)$ is a solution of the homogeneous Robin equation
\begin{equation}
\label{eq:homRobin}
\left\{
\begin{array}{rcll}
\partial_t \rho(t,u) & = & \partial^2_x \rho(t,u)\,,\quad 0<u<1\,,\;\;t>0\,,\\
 \partial_u \rho(t,0)&=&\mu \rho(t,0)\,,\;\;\;\;t>0\,, \\
 \partial_u \rho(t,1) &=&-\mu \rho(t,1)\,,\;\;\;\;t>0\,,\\
\rho(0,u) &=& f(u)\,,\;\;\;\;0\leq u\leq 1\,,
\end{array}
\right.
\end{equation}
where $f(u)=\rho_0(u)-\bar\rho(u)$.
Last equation corresponds to \eqref{eq:nonhRobin} with $\alpha=\beta=0$. 
This equation is suitable for Fourier methods, due to its linearity. Let us find the solutions of the eigenvalue problem
\begin{equation}
\label{eq:chile}
\left\{
\begin{array}{rcll}
\partial^2_u \phi^\mu(u) &=& -\lambda \phi^\mu(u)\,,\quad 0<u<1\,,\\
\partial_u \phi^\mu(0) &= & \mu \phi^\mu(0)\,,\\
  \partial_u \phi^\mu(1)&=& -\mu \phi^\mu(1)\,.\\
\end{array}
\right.
\end{equation}
We know that the solutions are going to be trigonometric functions. The real question is : what are the possible values of  the eigenvalues $\lambda$. For symmetry, let us try with $\phi$ of the form 
\[
\phi(u) = a\cos\big(\sqrt{\lambda} \big(u-\tfrac{1}{2}\big)\big).
\]
Then, the boundary conditions  at $u=0$ and $u=1$  are satisfied if and only if
\[
 \sqrt \lambda \sin\big( \tfrac{1}{2} \sqrt \lambda \big)=\mu \cos\big( \tfrac{1}{2} \sqrt \lambda\big),
\]
which can be written as the transcendental equation 
\[
\cot( \theta ) = \tfrac{2 \theta}{\mu},
\]
with $\lambda = 4 \theta^2$. This equation has a countable number of non-negative solutions $\{\theta_n; n \in \bb N\}$. If we number these solutions in increasing order, then they satisfy $$\theta_n \in \big[\pi(n-1),\pi(n-\sfrac{1}{2})\big]\,, \quad\mbox{for} \;\; n\geq 1\,.$$
Now we need to choose the normalizing constant $a$. This constant is fixed by the requirement $\int_0^1 \phi(u)^2\, du =1$. We have that
\[
\begin{split}
\int_0^1 \phi(u)^2\, du &= 2a^2 \int_0^{1/2}\!\!\!\!  \cos^2\big(\sqrt \lambda u\big) \,du
		= a^2 \int_0^{1/2}\!\!\!\! \big( 1+ \cos \big( 2 \sqrt \lambda u \big) \big)\,du= \tfrac{a^2}{2} \Big( 1 + \tfrac{\sin\sqrt \lambda}{\sqrt \lambda}\Big).
\end{split}
\]
Therefore, 
\[
a = \sqrt 2 \Big(  1 + \tfrac{\sin(\sqrt \lambda)}{\sqrt \lambda}\Big)^{-1/2}.
\]
Since the minimum of the function $\frac{\sin (u)}{u}$ is strictly smaller than $-1$, there  exists a constant $C$, not depending on $\mu$ or $n$, such that $a \leq C$ for any $\mu$ and any $n$. In other words, the functions $\{\phi_n\}_{n \in \bb N}$ are {\em uniformly bounded} by $C$. This remark will be important later on. The family of orthonormal functions $\{\phi_n; n \in \bb N\}$ constructed in this way forms a basis of the space of $L^2$-functions which are symmetric with respect to $u=1/2$. The other half of $L^2$ is obtained by taking functions $\psi$ of the form
\[
\psi(u) = b \sin\big( \sqrt{\lambda} \big(u-\tfrac{1}{2}\big)\big).
\]
In this case,   the boundary conditions  at $u=0$ and $u=1$  are satisfied if and only if\[
  - \sqrt \lambda \cos \big(\tfrac{1}{2} \sqrt \lambda \big)=\mu \sin \big( \tfrac{1}{2} \sqrt \lambda\big), 
\]
which corresponds to the transcendental equation
\[
\tan(\omega) =  - \tfrac{2\omega}{\mu}\,,
\]
for $\lambda = 4 \omega^2$.
This equation also has a countable number of solutions $\{\omega_n; n \in \bb N\}$. When numbered in increasing order, $\omega_n \in \big[ \pi(n-\sfrac{1}{2}), \pi n\big]$, for $n\in\bb N$. To make $\{\psi_n; n \in \bb N\}$ orthonormal, we have to choose
\[
b = \sqrt 2 \Big( 1 - \tfrac{\sin\sqrt \lambda}{\sqrt \lambda}\Big)^{-1/2}.
\]
Since $\omega_1 \geq \frac{\pi}{2}$ and the maximum of $\frac{\sin (u)}{u}$ outside $[-\sfrac{\pi}{2},\sfrac{\pi}{2}]$ is strictly smaller than $1$, we can take $C$ such that $b \leq C$, for any $\mu$ and any $n\in\bb N$. The sequence $\{\phi_n,\psi_n\}_{ n \in \bb N}$ forms an orthonormal basis of $L^2([0,1])$ of eigenvalues of the Laplacian operator with Robin boundary conditions. Note that the eigenvalues $\{\theta_n\}_{ n \in \bb N}$, $\{\omega_n\}_{ n \in \bb N}$ are interlaced: $\theta_n < \omega_n < \theta_{n+1}$ for any $n \in \bb N$. Therefore, we can rearrange the basis $\{\phi_n, \psi_n\}_{ n \in \bb N}$ as $\{\varphi_n\}_{ n \in \bb N}$ in such a way that $\partial_u^2 \varphi_n (u)= - \lambda_n \varphi_n(u)$ and $\lambda_n \in [\pi^2(n-1)^2, \pi^2 n^2]$.

Let $f \in L^2([0,1])$ be given. Define $\widehat{f}_n = \int_0^1 f(x) \varphi_n(x)\, dx$, then 
\begin{equation}
\label{peru}
\rho(t,u) = \sum_{n \in \bb N} \widehat{f}_n\, \varphi_n(u)\, e^{-\lambda_n t}
\end{equation}
is solution of  \eqref{eq:homRobin}.
Since $\{\varphi_n; n \in \bb N\}$ is uniformly bounded by $C$, a sufficient condition for continuity of $\rho(t,u)$ with respect to $u$ is that
\[
\sum_{n \in \bb N} |\widehat{f}_n| <+\infty. 
\]
This sum also bounds $\|\rho\|_\infty$.  But we need more regularity for $\rho$, then we need a stronger condition. Then we observe that there exists the  fourth space derivative of  $\rho$, which was defined in \eqref{peru},  under the condition
\begin{equation}\label{sum_peru}
\sum_{n \in \bb N} |\widehat{f}_n| \lambda_n^2 <+\infty.
\end{equation}
Since 
$\lambda_n\sim n^2$, the condition above implies that the solution $\rho$ of \eqref{eq:homRobin} is of class $ C^{1,4}$. 
%This condition is equivalent to 
%\[
%\sum_{n \in \bb N} |\widehat{f}_n| n^4 < +\infty.
%\]
Moreover, \eqref{sum_peru} implies
$
\|\rho\|_{{1,4}} \leq +\infty$.

The condition \eqref{sum_peru} holds if
 $f \in C^6$ and the support of $f$ is contained in the open interval $(0,1)$, because by integration by parts, we have
\[
|\widehat{f}_n| \leq \frac{C\|f^{(6)}\|_\infty}{\lambda_n^3}.
\]
From where we conclude that $\rho$ is of class $ C^{1,4}$ and moreover $\|\rho\|_{1,4}$ is uniformly bounded as a function of $\mu$.

\section*{Acknowledgements}

This project has received funding from the European Research Council (ERC) under  the European Union's Horizon 2020 research and innovative programme (grant agreement   No 715734).
A. N. thanks ``L'OR\' EAL - ABC - UNESCO Para Mulheres na Ci\^encia''.

\end{document}